\documentclass[a4paper,12pt]{article}

\usepackage{amsmath, amscd}
\usepackage{amsthm}
\usepackage{amssymb}
\usepackage{cancel}
\usepackage{color}
\usepackage{bbm}
\usepackage{mathtools}
\usepackage{graphicx}
\usepackage{enumerate}
\usepackage{t1enc}
\usepackage[latin2]{inputenc}
\newtheorem {theorem}{Theorem}
\newtheorem {lemma}[theorem]{Lemma}

\newtheorem {proposition}[theorem]{Proposition}

\def \RR {\mathbb R}

\def \ZZ {\mathbb Z}

\def\cA{\mathcal{A}}
\def\cB{\mathcal{B}}
\def\cC{\mathcal{C}}
\def\cD{\mathcal{D}}

\def\cF{\mathcal{F}}

\def\cH{\mathcal{H}}

\def\cK{\mathcal{K}}
\def\cL{\mathcal{L}}

\def\cN{\mathcal{N}}
\def\cO{\mathcal{O}}

\def\c S{\mathcal{S}}

\def\cU{\mathcal{U}}
\def\cV{\mathcal{V}}

\def\vareps{\varepsilon}

\newcommand{\probab}[1]{\ensuremath{\mathbf{P}\big(#1\big)}}
\newcommand{\expect}[1]{\ensuremath{\mathbf{E}\big(#1\big)}}
\newcommand{\var}[1]{\ensuremath{\mathbf{Var}\big(#1\big)}}

\newcommand{\condexpect}[2]{\ensuremath{\mathbf{E}\big(#1\bigm|#2\big)}}

\newcommand{\probabom}[1]{\ensuremath{\mathbf{P}_{\omega}\left(#1\right)}}
\newcommand{\expectom}[1]{\ensuremath{\mathbf{E}_{\omega}\left(#1\right)}}
\newcommand{\varom}[1]{\ensuremath{\mathbf{Var}_{\omega}\left(#1\right)}}

\newcommand{\condprobabom}[2]{\ensuremath{\mathbf{P}_{\omega}\left(#1\bigm|#2\right)}}
\newcommand{\condexpectom}[2]{\ensuremath{\mathbf{E}_{\omega}\left(#1\bigm|#2\right)}}

\newcommand{\ind}[1]{\ensuremath{\mathbbm{1}{(#1)}}}

\def\clap#1{\hbox to 0pt{\hss#1\hss}}

\DeclareMathOperator*{\wlim}{w\!\lim}

\def \Ordo {\cO}
\def\ordo{o}

\def\Ker{\mathrm{Ker}}
\def\Ran{\mathrm{Ran}}
\def\Dom{\mathrm{Dom}}

\renewcommand{\d}{\mathrm d}

\newcommand{\abs}[1]{\ensuremath |{#1}|}
\newcommand{\norm}[1]{\ensuremath \|{#1}\|}
\newcommand{\altnorm}[1]{\ensuremath |\!\Vert{#1}\Vert\!|}
\newcommand{\sprod}[2]{\ensuremath \langle{#1,#2}\rangle}

\def \wt {\widetilde}
\def\ol{\overline}
\def\ul{\underline}

\def \blue{\color{blue}}
\def \green {\color{green}}
\def \red {\color{red}}
\def \black {\color{black}}

\def \yellow {\color{yellow}}
\def \cyan {\color{cyan}}

\title{Central limit theorem for random walk in degenerate divergence-free random environment: $\mathcal H_{-1}$ reloaded with relaxed ellipticity}
\author{B\'alint T\'oth \\ {\tt R\'enyi Institute Budapest}}

\begin{document}

\maketitle

\begin{abstract}

\noindent
This paper  enhances the result of the work Kozma-T\'oth (2017) \cite{kozma-toth-17}. We prove the central limit theorem (in probability w.r.t. the environment) for the displacement of a random walker in divergence-free (or,  doubly stochastic) random environment, with \emph{substantially relaxed ellipticity} assumptions. Integrability of the reciprocal of the symmetric part of the jump rates is only assumed (rather than their boundedness, as in previous works on this type of RWRE). Relaxing ellipticity involves substantial changes in the proof, making it conceptually elementary in the sense that it does not rely on Nash's inequality in any disguise.



\end{abstract}

\section{Introduction}
\label{S:Introduction}

In the work  \cite{kozma-toth-17} the weak CLT (that is, in probability with respect to the environment) was established for random walks in doubly stochastic (or, divergence-free) random environments, under the conditions of 
($\imath$) \emph{strict ellipticity} assumed for the symmetric part of the drift field, and, 
($\imath\imath$) 
$\cH_{-1}(\abs{\Delta})$ assumed for the antisymmetric part of the drift field. 
\\
The proof relied on a martingale approximation (a la Kipnis-Varadhan) adapted to the \emph{non-self-adjoint} and \emph{non-sectorial} nature of the problem,  the two substantial technical components being:
\begin{enumerate}[$\circ$]
\item 
A functional analytic statement about the unbounded operator \emph{formally} written as $\abs{L+L^*}^{-1/2}(L-L^*)\abs{L+L^*}^{-1/2}$, where $L$ is the infinitesimal generator of the environment process,  as seen from the position of the moving random walker. 
\item 
A diagonal heat kernel upper bound which follows from Nash's inequality, valid only under the assumed \emph{strict ellipticity}.  
\end{enumerate}
The assumption of strict ellipticity, however, is conceptually restrictive and excludes relevant applications. 
In this paper we relax the strict ellipticity assumption, replacing it by an integrability condition on the reciprocals of the conductances. This can be done only by "de-Nashifying" the proof. On the other hand the functional analytic elements are refined. These changes are also of conceptual importance: the present proof (of a stronger result) does not invoke  a conceptually higher level (than the CLT) element like a local  heat kernel estimate. Altogether, it is conceptually simpler than that in \cite{kozma-toth-17}. (This was already demonstrated in \cite{toth-24} where the result of \cite{kozma-toth-17} was re-proved along a baby version of the arguments in the present paper.)

\smallskip
\noindent
To the best of our knowledge this is the first time  a CLT is proved for a random walk in a \emph{non-reversible} random environment without a strong ellipticity assumption. The integrability conditions imposed on the random jump rates seem to be close to optimal. 

\subsection{Preliminaries}
\label{ss:Preliminaries}

Let $(\Omega, \cF, \pi, (\tau_z:z\in\ZZ^d))$ be a probability space with an ergodic $\ZZ^d$-action. Denote by  ${\cN}:=\{k\in\ZZ^d: \abs{k}=1\}$ the set of unit elements generating $\ZZ^d$ as an additive group. These will serve as the set of elementary steps of a continuous time nearest neighbour random walk on $\ZZ^d$. 

Let $p:\Omega\to[0,\infty)^{\cN}$ satisfy ($\pi$-a.s.) the following \emph{bi-stochasticity} condition
\begin{align}
\label{bistoch}
\sum_{k\in{\cN}}p_k(\omega)
=
\sum_{k\in{\cN}} p_{-k}(\tau_k\omega), 
\end{align}
and $p:\ZZ^d\times\Omega\to [0,\infty)^{{\cN}}$ be its lifting to a random field over $\ZZ^d$,
\begin{align*}
p_k(x,\omega):=p_k(\tau_x\omega).
\end{align*}
(Throughout the paper, measurable functions $f:\Omega\to\RR$ and their lifting to a random field $f:\ZZ^d\times\Omega\to\RR$, defined as $f(x,\omega):=f(\tau_x\omega)$, will be denoted by the same symbol.) 

Given the random field  $p:\ZZ^d\times\Omega\to [0,\infty)^{{\cN}}$, define the continuous-time random walk in random environment (RWRE), $t\mapsto X(t)\in \ZZ^d$ as the Markovian nearest neighbour  random walk with jump rates
\begin{align}
\label{rwre}
\condprobabom{X(t+dt)= x+k}{X(t)=x} = 
p_k(x, \omega)\,dt + \Ordo((dt)^2),
\end{align}
and initial position $X(0)=0$. 

We use the notation $\probabom{\cdot}$, $\expectom{\cdot}$  and $\varom{\cdot}$ for \emph{quenched} probability, expectation and variance. That is: probability, expectation, and variance with respect to the distribution of the random walk $X(t)$, \emph{conditioned on fixed environment $\omega\in\Omega$}. The notation $\probab{\cdot}:=\int_\Omega\probabom{\cdot} {\d}\pi(\omega)$, $\expect{\cdot}:=\int_\Omega\expectom{\cdot} {\d}\pi(\omega)$ and $\var{\cdot}:=\int_\Omega\varom{\cdot} {\d}\pi(\omega) + \int_\Omega\expectom{\cdot}^2 {\d}\pi(\omega) - \expect{\cdot}^2$ will be reserved for \emph{annealed} probability, expectation and variance. That is: probability,  expectation and variance with respect to the random walk trajectory $X(t)$ \emph{and} the environment $\omega$, sampled according to the distribution $\pi$. 

The \emph{environment process} (as seen from the position of the random walker) is, $t\mapsto \eta_t\in\Omega$, defined as 
\begin{align}
\label{envproc}
\eta_t:= \tau_{X_t}\omega. 
\end{align}
This is a pure jump Markov process on the state space $\Omega$. The infinitesimal generator of its Markovian semigroup is 
\begin{align}
\label{inf-gen-op}
Lf(\omega)
&
=
\sum_{k\in{\cN}} p_k(\omega)(f(\tau_k\omega)-f(\omega)).
\end{align}
The linear operator $L$  is well defined for all measurable functions $f:\Omega\to  \RR$, not just a formal expression. 

It is well known (and easy to check, see e.g. \cite{kozlov-85}) that bi-stochasticity \eqref{bistoch} of the jump rates $p$ is equivalent to stationarity (in time) of the a priori distribution $\pi$ of the environment process $t\mapsto \eta_t\in\Omega$. Moreover, under the conditions \eqref{bistoch} and 
{\yellow \eqref{weakell}}  
(see below) spatial ergodicity of $(\Omega, \cF, \pi, (\tau_z:z\in\ZZ^d))$ also implies time-wise ergodicity of the (time-wise) stationary environment process process $t\mapsto \eta_t\in(\Omega, \cF,\pi)$. See \cite{kozlov-85} for details. Hence it follows that under these conditions the random walk $t\mapsto X(t)$ will have stationary and ergodic \emph{annealed} increments. Though, in the annealed setting the walk is obviously not Markovian.

It is convenient to separate the symmetric and antisymmetric parts of the jump rates: 
\begin{align}
\label{jump rates p s b}
&
p_k(\omega)
=
s_k(\omega)+b_k(\omega),
\quad
&&
s_k(\omega)
:=
\frac{p_k(\omega)+p_{-k}(\tau_k\omega)}{2},
\quad
&&
b_k(\omega)
:=
\frac{p_k(\omega)-p_{-k}(\tau_k\omega)}{2},
\end{align}
and note that 
\begin{align}
\label{s b symm}
& 
s_k(\omega)=s_{-k}(\tau_k\omega)
\geq0, 
&&
b_k(\omega)=-b_{-k}(\tau_k\omega), 
\end{align}
and \eqref{bistoch} is equivalent to 
\begin{align}
\label{b is divfree}
\sum_{k\in\cN} b_k(\omega)=0. 
\end{align}
We assume for now  \emph{weak ellipticity} and \emph{finiteness} of the conductances: $\pi$-a.s. for all $k\in\cN$, 
\begin{align}
\label{weakell}
0< 
s_k(\omega)
<\infty.
\end{align}
Later somewhat stronger integrability  conditions  will be imposed (see \eqref{condintegr} further below). 

Accordingly, we decompose the infinitesimal generator into  Hermitian, and anti-Her\-mi\-tian parts (with respect to the stationary measure $\pi$) as 
\[
L=-S+A
\]
where 
\begin{align}
\label{S-A-ops}
&
Sf(\omega)
:=
-
\sum_{k\in{\cN}} 
s_k(\omega)(f(\tau_k\omega)-f(\omega)), 
&&
Af(\omega)
:=
\sum_{k\in{\cN}} 
b_k(\omega)
(f(\tau_k\omega)-f(\omega)). 
\end{align}
The right hand sides of \eqref{inf-gen-op} and \eqref{S-A-ops} make perfect sense for arbitrary measurable functions $f:\Omega\to \RR$. So, disregarding for the time being topological issues in function spaces, the linear operators $L$, $S$, and $A$ are well defined over the space of all (classes of equivalence of) measurable functions $f:\Omega\to\RR$ (w.r.t. the probability measure $\pi$).

The symmetric part of the jump rates (lifted to $\ZZ^d$ as $s_k(x,\omega):=s_k(\tau_x\omega)$),  
\begin{align}
\label{s-cond}
s_{-k}(x+k,\omega)
=
s_k(x,\omega)
>0 
\end{align}
defines a collection positive \emph{random conductances} on the unoriented edges $\{x,x+k\}$. The weak ellipticity condition \eqref{s-cond} means that all edges of the lattice $\ZZ^d$ are passable for the random walker, at least in one direction. Integrability conditions on the conductances $s_k$ and their reciprocals $s_k^{-1}$ will be imposed later (see \eqref{condintegr} below). Actually, it will be convenient to work with the variables 
\begin{align*}
r_k(\omega):=\sqrt{s_k(\omega)}. 
\end{align*}

Regarding the antisymmetric parts of the jump rates (lifted to $\ZZ^d$ as $b_k(x,\omega):=b_k(\tau_x\omega)$), we have 
\begin{align}
\label{b-flow}
b_{-k}(x+k,\omega)
=
-
b_k(x,\omega).
\end{align}
That is, they define a \emph{random flow} (or, a \emph{lattice vector field}) on the oriented edges $(x,x+k)$ of $\ZZ^d$. In terms of these variables the bi-stochasticity \eqref{bistoch} reads as \emph{sourcelessness}, or \emph{divergence-freeness} of the flow: $\pi$-a.s., for all $x\in\ZZ^d$
\begin{align}
\label{divfree}
\sum_{k\in{\cN}} b_k(x, \omega)=0
\end{align}
Obviously, the former dominate the latter: 
\begin{align}
\label{domin}
\abs{b_k(\omega)}\leq s_k(\omega).
\end{align}
The extreme case, when $\pi$-almost-surely, for all $k\in{\cN}$, $s_k(\omega)=\abs{b_k(\omega)}$, 
will be called \emph{totally asymmetric}. In this case every edge of the lattice $\ZZ^d$ is passable in exactly one of the two possible directions.

Throughout this paper we assume that the divergence-free flows/vector fields $b_k$ are given in \emph{divergence form}. That is, there exists $h:\Omega\to\RR^{{\cN}\times{\cN}}$, with the following ($\pi$-a.s.) symmetries
\begin{align}
\label{h-tensor}
h_{k,l}(\omega)
=
-h_{-k,l}(\tau_k\omega)
=
-h_{k,-l}(\tau_l\omega)
=
-h_{l,k}(\omega)
\end{align}
so that 
\begin{align}
\label{b-is-curl-of-h}
b_{k}(\omega)
=
\sum_{l\in{\cN}} h_{k,l}(\omega)
=
\frac12 \sum_{l\in{\cN}} (h_{k,l}(\omega)-h_{k,l}(\tau_{-l}\omega)).
\end{align}
It is straightforward that \eqref{h-tensor} and \eqref{b-is-curl-of-h} imply \eqref{b-flow}. The converse, however, is not true. It is actually a very subtle issue to determine whether a divergence free vector field (as e.g. $b_k(x, \omega)$ given in \eqref{jump rates p s b}) can be written in this so-called divergence-form or not. More on this in the Appendix.

The symmetry conditions \eqref{h-tensor} mean that the lifted field $h:\ZZ^d\times\Omega\to\RR^{{\cN}\times{\cN}}$, $h_{k,l}(x,\omega):= h_{k,l}(\tau_x\omega)$, 
is a \emph{random function of the oriented plaquettes} of $\ZZ^d$, that is,  a stationary (with respect to spatial translations) \emph{stream tensor}. Note that in dimensions $d=2$ and $d=3$, the stream fields $(x,k,l)\mapsto h_{k,l}(x, \omega) := h_{k,l}(\tau_x \omega)$ are identified with a \emph{scalar height function} on $\ZZ^{2*}:=\ZZ^2+(1/2,1/2)$, respectively, a \emph{flow/vector field} on  $\ZZ^{3*}:=\ZZ^3+(1/2,1/2,1/2)$.

In \cite{kozma-toth-17} the weak CLT (that is, in probability with respect to the environment) was established for the diffusively scaled displacement $T^{-1/2}X(T)$, as $T\to\infty$, under the conditions 
\begin{align}
\label{strong ell}
s_k+(s_k)^{-1} \in\cL_\infty
\end{align}
imposed on the conductances and 
\begin{align}
\label{old H-1}
h_{k,l}\in \cL_2
\end{align}
imposed on the stream tensor. Throughout this paper $\cL_p:=\cL_p(\Omega, \pi)$, with $1\leq p\leq \infty$.
Condition \eqref{old H-1} is equivalent to $b_k\in\cH_{-1}(\abs{\Delta})$. That is, the antisymmetric part of the jump rates belong to the $\cH_{-1}$-space of the lattice Laplacian acting on $\cL_2$.

The upper bound in \eqref{strong ell} could be relaxed to $s_k\in\cL_1$ without effort, and without altering the proof in \cite{kozma-toth-17}. 
The lower bound (\emph{strong ellipticity}) is, however, a serious issue. The main goal of this paper is to relax it to an integrability condition imposed on $(s_k)^{-1}$. 

\subsection{Integrability conditions on the jump rates}
\label{ss: Assumptions}

Regarding the {\it symmetric parts / conductances},  we assume that for all $k\in{\cN}$,
\begin{align}
\label{condintegr}
s_k+(s_k)^{-1} \in\cL_1.
\end{align}
Or, equivalently
\begin{align}
\label{r-l2}
&
r_k
\in\cL_2, 
\\[8pt]
\label{rrec-l2}
&
(r_k)^{-1}
\in\cL_2.
\end{align}
Regarding the {\it antisymmetric parts / flows}, note first that due to \eqref{domin} and 
{\red \eqref{r-l2}} 
we also have for all $k\in\cU$
\begin{align*}
b_k\in\cL_1. 
\end{align*}
Beside
{\red \eqref{r-l2}}
and 
{\blue \eqref{rrec-l2}} 
we also assume that for all $k,l\in{\cN}$,
\begin{align}
\label{new H-1}
(r_{l})^{-1} h_{k,l}
\in\cL_2. 
\end{align}
Note that assuming strong ellipticity, $(r_k)^{-1}\in\cL_{\infty}$, \eqref{new H-1} becomes \eqref{old H-1}.  Furthermore, 
{\red \eqref{r-l2}}
\& 
{\green \eqref{new H-1}}
jointly also imply 
\begin{align}
\label{h-l1}
h_{k,l}
\in\cL_1, 
\end{align}

\smallskip
\noindent
{\bf Summarizing:} 
We make the \emph{structural} assumption \eqref{h-tensor}\&\eqref{b-is-curl-of-h} and the \emph{integrability} assumptions 
{\red \eqref{r-l2}}, 
{\blue \eqref{rrec-l2}}, 
{\green \eqref{new H-1}}. 
We will be explicit about exactly which integrability assumption/condition is used in each step of the forthcoming arguments.

\subsection{The CLT}
\label{ss: The CLT}

The local \emph{quenched} drift of the random walk is
\begin{align}
\label{local-drift}
\condexpectom{dX(t)}{X(s): 0\le s\le t} 
= 
\left(\phi(\eta_t) + \psi(\eta_t)\right)\, dt +\ordo(dt).
\end{align}
where $\phi, \psi : \Omega\to\RR^d$ are 
\begin{align}
\label{phi-and-psi}
& 
\phi(\omega):= \sum_{k\in{\cN}} k s_k(\omega),
&&
\psi(\omega):= \sum_{k\in{\cN}} k b_k(\omega).
\end{align}
Due to \eqref{s-cond}, respectively, \eqref{b-flow} we have
\begin{align}
\label{drift fields}
\phi_i(\omega)
=
s_{e_i}(\omega)-s_{e_i}(\tau_{-e_i}\omega),
&&
\psi_i(\omega)
=
b_{e_i}(\omega)+b_{e_i}(\tau_{-e_i}\omega). 
\end{align}
In particular, from 
{\red\eqref{r-l2}} 
we have $\phi\in\cL_1$ and 
\begin{align}
\label{no-symm-drift}
\int_{\Omega} \phi(\omega)\,d\pi(\omega)
=0.
\end{align}
This does not hold a priori for the drift term $\psi$ coming from the  divergence-free antisymmetric part. However, from \eqref{b-is-curl-of-h} 
and 
{\cyan \eqref{h-l1}} 
we have $\psi\in\cL_1$ and 
\begin{align}
\label{no-antisymm-drift}
\int_{\Omega} \psi(\omega)\,d\pi(\omega)
=0.
\end{align}
From ergodicity under $\pi$ of the environment process $t\mapsto\eta_t$, and from 
\eqref{local-drift}, \eqref{no-symm-drift} and \eqref{no-antisymm-drift} the strong law of large numbers for the displacement of the random walker readily follows. 

\begin{proposition}
[Strong Law of Large Numbers]
\label{prop: slln}
\phantom{}
\\
Assuming 
{\red \eqref{r-l2}}
and 
{\cyan \eqref{h-l1}}, 
\begin{align*}
\lim_{t\to\infty} t^{-1} X(t) =0,
\qquad
\mathrm{a.s.}
\end{align*}
\end{proposition}

\noindent
The main result of this paper is:

\begin{theorem}
[Main theorem: CLT]
\label{thm: main}
\phantom{}
\\
Assume the integrability conditions
{\red \eqref{r-l2}}, 
{\blue \eqref{rrec-l2}}, 
{\green \eqref{new H-1}}. 
Then the displacement $t\mapsto X(t)$ of the random walk can be decomposed as 
\begin{align*}
X(t)=Y(t)+Z(t)
\end{align*}
so that the following limits hold as $N\to \infty$. 

\begin{enumerate} [(i)]

\item 
For $\pi$-almost all $\omega\in\Omega$, $t\mapsto Y(t)$ is a nondegenerate square integrable martingale whose increments are stationary and ergodic in the annealed setting. Thus, due to the Martingale Invariance Principle (cf. \cite{mcleish-74}), for $\pi$-almost all $\omega\in\Omega$,  
\begin{align*}
N^{-1/2} Y(N\cdot) \Rightarrow W_\sigma (\cdot), 
\end{align*}
in $D([0,1])$ under the quenched probability measure $\probabom{\dots}$, where $W_\sigma(\cdot)$ is a Brownian motion with finite and non-degenerate covariance $\sigma^2$.

\item 
For any $t\geq0$ and $\delta>0$
\begin{align*}
\probab{N^{-1/2}\abs{Z(Nt)}>\delta} \rightarrow 0.
\end{align*}

\end{enumerate}

\end{theorem}

\smallskip
\noindent
{\bf Remark.}
We should emphasize that this is much stronger than an annealed limit theorem. The invariance principle for the displacement of the random walk
\begin{align}
\label{IP}
N^{-1/2} X(N\cdot) \Rightarrow W_\sigma (\cdot), 
\qquad 
\text{ as } 
N\to\infty, 
\end{align}
holds \emph{in probability} (and not averaged out) with respect to the the distribution of the environment, $(\Omega, \pi)$. In other words, the invariance principle \eqref{IP} holds $\pi$-a.s. (that is, in a quenched sense) along sequences $N\to\infty$ increasing sufficiently fast. 

\subsection{Examples}
\label{ss: Examples}

We give some concrete constructive examples -- not covered by earlier works -- where the arguments of this paper work and provide a CLT (in the sense of Theorem \ref{thm: main}). 

\subsubsection{Adding stream term to random conductances}
\label{sss: Adding stream term to random conductances}

Let $\wt s_k(x,\omega)=\wt s_k(\tau_x\omega)$, $k\in\cN$, $x\in\ZZ^d$, be random conductances as in \eqref{s-cond} and assume that the integrability conditions \eqref{condintegr} hold for them. It has been known since \cite{demasi-et-al-89} that the random walk among these conductances obeys the CLT as stated in Theorem \ref{thm: main}. It is a natural question to ask how can one modify this random environment \emph{by adding antisymmetric, nonreversible part to the jump rates} so that the nondegenerate CLT remains valid. Here is an answer.

Let $h_{k,l}(x, \omega)=h_{k,l}(\tau_x\omega)$, $k,l\in\cN$, $x\in\ZZ^d$, be a stream tensor, as in \eqref{h-tensor} (defined on the same probability space, jointly with the conductances) and $b_k(x,\omega)= b_k(\tau_x\omega)$, $k\in\cN$, $x\in\ZZ^d$, be the divergence-free flow/vector field given by \eqref{b-is-curl-of-h}.  Let now 
\begin{align*}
p_k(\omega)
:=
\wt s_k(\omega) + 2 b_k(\omega)_+.
\end{align*}
(That is $s_k(\omega):=\wt s_k(\omega) + \abs{b_k(\omega)}$.)
It will suffice to assume (beside \eqref{condintegr} holding for $\wt s_k$) that
\begin{align*}
\frac{\abs{h_{k,l}}}{\sqrt{\wt s_k}}
\in\cL_2.
\end{align*}
In particular, if $\big(\wt s_k(x)\big)_{x\in\ZZ^d, k\in\cN}$ at one hand and $\big(h_{k,l}(x)\big)_{x\in\ZZ^d, k,l\in\cN}$ on the other, are \emph{independent} then it suffices that $h_{k,l}\in\cL_2$. 

If this independence does not hold than, $(\wt s_k)^{-1}\in\cL_\beta$ and $h_{k,l}\in\cL_\alpha$, with $2/\alpha + 1/\beta = 1$ will suffice. $\alpha = \infty$, $\beta=1$ means a perturbation of the random conductance model satisfying \eqref{condintegr} (without strong ellipticity assumed!) with a divergence free drift field which arises as the curl of a \emph{bounded} stream tensor. The case $\alpha=2$, $\beta=\infty$ means divergence free drift field arising as the curl of an $\cL_2$ stream tensor - that is, in $\cH_{-1}(\abs{\Delta})$ of the lattice-Laplacian, and strongly elliptic, fully covered in \cite{kozma-toth-17}. 

\subsubsection{Two totally asymmetric examples}
\label{sss: Totally asymmetric examples}

Let $h_{k,l}(x, \omega)=h_{k,l}(\tau_x\omega)$, $k,l\in\cN$, $x\in\ZZ^d$, be a stream tensor, as in \eqref{h-tensor},  $b_k(x,\omega)= b_k(\tau_x\omega)$, $k\in\cN$, $x\in\ZZ^d$, be the divergence-free flow/vector field given by \eqref{b-is-curl-of-h} and let the jump rates be 
\begin{align*}
p_k(x, \omega)
:=
2 b_k(x, \omega)_+.
\end{align*}
(That is $s_k(\omega):=\abs{b_k(\omega)}$.) Note that each edge $(x,x+k)$ is passable in exactly one of the two directions. It is straightforward that the integrability conditions 
{\red \eqref{r-l2}}, 
{\blue \eqref{rrec-l2}} 
and 
{\green \eqref{new H-1}} 
follow from $h_{k,l}\in\cL_\alpha$ and $\abs{b_k}^{-1}\in\cL_\beta$ with $2/\alpha + 1/\beta = 1$. $\alpha=\infty$, $\beta=1$ means bounded stream tensor (and, as a consequence, bounded jump rates) and integrable reciprocal jump rates (no strong ellipticity). The case $\alpha=2$, $\beta=1$ means strong ellipticity and $\cL_2$ stream tensor - covered in \cite{kozma-toth-17}.  

Another totally asymmetric example is the following non-elliptic \emph{perturbation} (though, not a small perturbation) of elliptic cases covered in \cite{kozma-toth-17}. 
Let $\wt h_{k,l}(x, \omega)=\wt h_{k,l}(\tau_x\omega)$, $k,l\in\cN$, $x\in\ZZ^d$, be a stream tensor, as in \eqref{h-tensor},  $\wt b_k(x,\omega)= \wt b_k(\tau_x\omega)$, $k\in\cN$, $x\in\ZZ^d$, be the corresponding divergence-free flow/vector field given by \eqref{b-is-curl-of-h}. Assume $\wt h_{k,l}\in \cL_2$ and $(\wt b_k)^{-1}\in\cL_\infty$ - that is, the flow/vector field $\wt b_k(x, \omega)$ is assumed to be strongly elliptic. Without loss of generality, assume $\inf \abs{\wt b_k}= 1$ $\pi$-a.s. As concrete examples think about either the $\binom{2d}{d}$-model on $\ZZ^d$, with $d\geq3$,  or the randomly oriented Manhattan lattice on $\ZZ^d$ with $d\geq4$. For details see \cite{kozma-toth-17}, \cite{toth-18b}. In addition, let $(\xi_{k,l}(x))_{x\in\ZZ^d, k,l,\in\cN}$ be i.i.d. random variables modulo the tensorial symmetries in \eqref{h-tensor}, and  also independent of the collection $(\wt h_{k,l}(x))_{x\in\ZZ^d, k,l,\in\cN}$, distributed as $\xi_{kl}(x)\sim {\tt UNI} [-a,+a]$, with the value of $a$ to be specified. Finally, let $h_{k,l}(x)=\wt h_{k,l}(x)+\xi_{k.l}(x)$, $b_k(x)$ given by \eqref{b-is-curl-of-h} and $p_k(x)=2b_k(x)$. If $0\leq a < 1/(2(d-1))$ then the model is still strongly elliptic with $s_k=\abs{b_k}\geq 1-2(d-1)a>0$ and as such it is covered by \cite{kozma-toth-17}. If $a=1/(2(d-1))$ then the model is not elliptic and the integrability conditions 
{\red \eqref{r-l2}}, 
{\blue \eqref{rrec-l2}},  
{\green \eqref{new H-1}} 
hold. If $a>1/(2(d-1))$ then conditions 
{\blue \eqref{rrec-l2}},  
{\green \eqref{new H-1}} 
fail just marginally. 

\subsubsection{RWRE with stationary density}
\label{sss: RWRE with stationary density}

Let $p_k(x,\omega)=p_k(\tau_x\omega)$, $x\in\ZZ^d$, $k\in\cN$, be the jump rates of a random walk in random environment, as in \eqref{rwre}. If the bistochasticity condition \eqref{bistoch} fails to hold that the environment process $t\mapsto \eta_t$ cf \eqref{envproc} is not time-stationary under the a priori measure $\pi$. The equation for a stationary Radon-Nikodym derivative $\varrho\in \cL_{1}$ is, cf \cite{kozlov-85}, 
\begin{align}
\label{R-N eq}
\sum_{k\in\cN} \varrho(\tau_k\omega) p_{-k}(\tau_k\omega)
=
\sum_{k\in\cN} \varrho(\omega) p_{k}(\omega).
\end{align}
It is very natural to expect that solving \eqref{R-N eq} is the first step towards a CLT for the generic RWRE.  This is a very hard equation, there seem to be no general methods to handle it. Few exceptions (when it is solved) are the one-dimensional cases studied in \cite{szasz-toth-84} where a continued fraction expansion is proved to work, the \emph{balanced rwre} case studied  first on \cite{lawler-82} where a hands-on solution is constructed for that particular model, and the \emph{rw in Dirichlet-distributed random environment} studied in \cite{sabot-13} where the solution of equation \eqref{R-N eq} is constructed via deep relations of this particular model with the reinforced random walk problem. The CLT question is settled (positively) in \cite{lawler-82} and \cite{szasz-toth-84}, and left widely open in \cite{sabot-13}, and any other case when \eqref{R-N eq} could possibly be solved. 

Assume that \eqref{R-N eq} is solved with $\varrho\in \cL_1$ and define a new RWRE $t\mapsto \wt X(t)$ with jump rates 
\begin{align}
\label{new-jump-rates}
\wt p_k(\omega):=  \varrho(\omega) p_{k}(\omega). 
\end{align}
Obviously, the jump rates $\wt p_k(x,\omega)$ are bistochastic, as in \eqref{bistoch}. 

The random walks $X(t)$ and $\wt X(t)$ differ only in the holding times between the successive steps: at site $x\in\ZZ^d$ they wait for ${\tt EXP}(\sum_{k\in\cN} p_k(\tau_x\omega))$-, respectively for ${\tt EXP}(\sum_{k\in\cN} \wt p_k(\tau_x\omega))$-distributed holding times and jump to the neighbouring site $x+k$ with the same probabilities $(\sum_{k\in\cN} p_k(x, \omega))^{-1}p_k(x, \omega) = (\sum_{k\in\cN} \wt p_k(x, \omega))^{-1}\wt p_k(x, \omega)$. So, proving the CLT for the walks $X(t)$ and $\wt X(t)$ are related through laws of large numbers for the sums of the holding times in both cases. However, these LLNs hold due to the ergodic theorem. So, it remains to establish the CLT for the walk $\wt X(t)$. That is, to check the structural conditions \eqref{h-tensor}, \eqref{b-is-curl-of-h} and the integrability conditions 
{\red \eqref{r-l2}}, 
{\blue \eqref{rrec-l2}}, 
{\green \eqref{new H-1}} 
for the jump bistochastic rates $\wt p_k(x,\omega)$ defined in \eqref{new-jump-rates}. 
The truly difficult part is checking the structural conditions \eqref{h-tensor}, \eqref{b-is-curl-of-h}. This is work for the future. Given  \eqref{h-tensor}, \eqref{b-is-curl-of-h},  the integrability conditions could be controlled by the appropriate choice of the parameters of the Dirichlet-distributed environment.

\subsection{Related earlier works}
\label{ss: Related earlier works}

Since random walks and/or diffusions in \emph{divergence-free} random drift field are not merely mathematical toys, metaphoric so-called "models", but are directly motivated by the true physical problem of self-diffusion in turbulent and incompressible steady state flow, these problems have a notorious and long history. The historic background in the context of strongly elliptic (cf \eqref{strong ell}) environments was presented in sufficient details in section 1.6 of the survey paper \cite{toth-18b}.  Here I just mention the most important forerunning works - in chronological order - are 
\cite{osada-83},  
\cite{kozlov-85}, 
\cite{oelschlager-88}, 
\cite{fannjiang-papanicolaou-96}, 
\cite{fannjiang-komorowski-97}, 
\cite{komorowski-olla-02}, 
\cite{komorowski-olla-03a}, 
\cite{komorowski-olla-03b}, 
\cite{deuschel-kosters-08}, 
\cite{toth-valko-12}, 
\cite{komorowski-landim-olla-12}, 
\cite{kozma-toth-17}, 
\cite{ledger-toth-valko-18}, 
\cite{toth-18a}. 
The list is by no means exhaustive. Note, however, that in all these works \emph{strong ellipticity} cf \eqref{strong ell} (or, an equivalent condition for the diffusion settings) is assumed. 

As for the non-elliptic setting, much effort has been put on relaxing the strong ellipticity condition \eqref{strong ell} in the context of pure conductance models, where $p_k=s_k$, $b_k=0$. In this setting the CLT of Theorem \ref{thm: main} under conditions \eqref{r-l2} and \eqref{rrec-l2} has been established already in \cite{demasi-et-al-89}. The strongest results in the fully quenched setting (that is, CLT for the random walk in $\pi$-almost all environment)  have been established in \cite{biskup-11} ($d=2$) and \cite{bella-schaffner-20} ($d\geq3$) where the quenched CLT is proved under the \emph{dimension dependent} conditions
\begin{align*}
s_k\in\cL_p, 
\quad
(s_k)^{-1} \in \cL_q, 
\qquad
\frac{1}{p}+ \frac{1}{q} < \frac{2}{d-1}. 
\end{align*}
I have no knowledge of any such result with relaxed ellipticity for the non-reversible cases, where $b_k\not\equiv0$. As seen from \cite{kozma-toth-17}, \cite{toth-18a}, and also from this paper, the extension is by no means straightforward. 

\subsection{Blueprint of the proof and structure of the paper}
\label{ss: Blueprint and structure of the paper}

We follow  the usual route. Based on \eqref{local-drift} we decompose the displacement of the random walker as
\begin{align}
\label{mart-decomp}
X(t)
=
\underbrace{X(t)-\int_0^t (\phi(\eta_s) + \psi(\eta_s))\,ds}_{\displaystyle M(t)}
+
\underbrace{\int_0^t \phi(\eta_s))\,ds}_{\displaystyle I(t)}
+
\underbrace{\int_0^t \psi(\eta_s))\,ds}_{\displaystyle J(t)}.
\end{align}
Here $t\mapsto M(t)$ is a martingale with respect to the quenched measures $\probabom{\cdot}$, with stationary and ergodic annealed (i.e., w.r.t. the measure $\probab{\cdot}$) increments. 

In 
Section \ref{s: Spaces and operators}, 
after some functional analytic preliminaries 
(sections \ref{ss: Basic spaces and operators} and \ref{ss: The infinitesimal generator}) 
we give a detailed exposition of the natural $\cH_{\pm}$-spaces related to the self-adjoint part of the infinitesimal generator (section \ref{ss: "Sobolev spaces"}). 

Section \ref{s: Diffusive bounds}
is devoted to proving annealed diffusive bounds on the displacement $t\mapsto X(t)$ of the random walk. 
We show that under the conditions 
{\red \eqref{r-l2}} and 
{\green \eqref{new H-1}} 
imposed, the drift fields $\phi, \psi:\Omega\to \RR^d$ are (component-wise) in $\cH_{-}$ (section \ref{ss: Upper}). Due to Varadhan's $\cH_{-}$-bound this also implies an annealed  diffusive upper bound on  the last two integrals on the right hand side of \eqref{mart-decomp} and thus, on the displacement (section \ref{ss: Upper}). 
Relying on 
{\blue \eqref{rrec-l2}}
a diffusive lower bound is also established (section \ref{ss: Lower}). This is subtler than corresponding arguments in the random conductance case. 

Theorem \ref{thm: main} is proved in Section \ref{s: Proof of Theorem {main}}. An efficient martingale approximation of the integral terms $t\mapsto I(t)+J(t)$ in \eqref{mart-decomp} is done by applying the non-reversible (non-self-adjoint) variant of the Kipnis-Varadhan theorem, cf \cite{toth-86a}, \cite{horvath-toth-veto-12}, \cite{toth-13}. However, since there is no natural grading of the infinitesimal generator $L$ acting on the Hilbert space $\cL_2$, the Graded Sector Condition cannot be applied. Instead, we employ a stronger version of the  \emph{relaxed sector condition} introduced in \cite{horvath-toth-veto-12} and already used in \cite{kozma-toth-17}. (The general theory -- in a somewhat abstract form -- is summarized in section \ref{ss: Kipnis-Varadhan}.) In this order we have to prove that the (unbounded) operator \emph{formally} defined as $B:=S^{-1/2} A S^{-1/2}$ (where $S:=-(L+L^*)/2$, $A:=-(L-L^*)/2$) makes sense as a  \emph{skew-self-adjoint operator} over $\cL_2$. This is proved in subsection \ref{ss: Checking RSC}, concluding the proof of the main result. This is the novel part of the proof. In \cite{kozma-toth-17} the argument for proving $B^*=-B$ heavily relied on Nash's inequality which, on its turn assumes strong ellipticity \eqref{strong ell} of the jump rates. Relaxing strong ellipticity to the much weaker condition \eqref{weakell} (in particular, 
{\blue \eqref{rrec-l2}}) 
required a new \emph{de-Nashified} proof which turns out to be conceptually  simpler than the one in \cite{kozma-toth-17}. As a bonus we also obtain existence and uniqueness of the 
so-called \emph{harmonic coordinates} (section \ref{ss: Harmonic coordinates - existence and uniqueness}), thus opening the way to a quenched CLT - to be completed in future work. 

In the Appendix we state a Helmholtz-type theorem shedding light on the problem of when and how can a divergence-free flow/vector field expressed in divergence form, as in \eqref{h-tensor}\&\eqref{b-is-curl-of-h}.

\section{Spaces and operators}
\label{s: Spaces and operators}

\subsection{Basic spaces and operators}
\label{ss: Basic spaces and operators}
We define various function spaces (over  $(\Omega, \pi)$) and linear operators acting on them. With usual abuse we denote \emph{classes of equivalence} of $\pi$-a.s. equal measurable functions simply as functions. Let the space of \emph{scalar}-, \emph{vector}-, \emph{rotation-free vector}- and \emph{divergence-free vector}- fields be 
\begin{align*}
\cL
&
:=
\{f:\Omega\to\RR: 
f \text{ is } \cF\text{-measurable}\}
\\[10pt]
\cV
&
:=
\{u:\Omega\to\RR^{{\cN}}: 
u_k\in \cL, \ \ 
u_k(\omega)+u_{-k}(\tau_k\omega)=0, \ \ k\in{\cN}, 
\ \ \pi\text{-a.s.}\}
\\[10pt]
\cU
&
:=
\{u\in\cV: 
u_k(\omega)+u_l(\tau_k\omega)=u_l(\omega)+u_k(\tau_l\omega),  \ \ 
k,l\in{\cN},  \ \ \pi\text{-a.s.}\}.
\end{align*}
These are linear spaces (over $\RR$) with no norm or topology endowed on them  yet.  We call these spaces these names for the obvious reason that their liftings
\begin{align*}
&
f(x,\omega):= f(\tau_x\omega)
\ \ \ 
(f\in\cL)
&&
u(x,\omega):= u(\tau_x\omega)
\ \ \ 
(u\in\cV)
\end{align*}
are  translation-wise ergodic scalar- ,  respectively, vector-fields over $\ZZ^d$. 

The linear operators $\partial_k, R_k, H_{k,l}: \cL\to\cL$, $k,l\in{\cN}$, defined below  on the whole space $\cL$ as their domain, will be the basic primary objects used in constructing more complex operators.
\begin{align}
\label{basic-ops}
\begin{aligned}
& 
T_k f(\omega)
:=
f(\tau_k\omega),
&&
\partial_k 
:=
T_k-I, 
\\[10pt]
&
R_kf(\omega)
:=
r_k(\omega) f(\omega),
\hskip2cm
&& 
H_{k,l}f(\omega)
:=
h_{k,l}(\omega) f(\omega)
\end{aligned}
\end{align}
Using these basic operators we further define
\begin{align}
\notag
&
\nabla:\cL\to\cU,
&&
(\nabla f)_k
:=
\partial_k f
\\[10pt]
\notag
&
\nabla^*:\cV\to \cL, 
&&
\nabla^* u
:= 
\sum_{k\in{\cN}}u_k
=
-
\frac12
\sum_{k\in{\cN}} \partial_{-k} u_k
\\[10pt]
\label{R-op}
&
R:\cV\to\cV,
&&
(Ru)_k
:=
R_ku_k
\\[10pt]
\label{H-op}
&
H:\cV\to\cV,
&&
(Hu)_k
:=
\frac12
\sum_{l\in{\cN}} 
\,H_{k,l}\,
(T_{k} + I ) 
u_l
\\
\notag
& 
&& 
\phantom{(Hu)_k
:}
=
\frac14
\sum_{l\in{\cN}} 
(T_{-l} +  I ) 
\,H_{k,l}\,
(T_{k} + I ) 
u_l
\end{align}
These operators are well defined on the whole spaces given as their respective domains. For the time being the superscript $^*$ is only notation. It will later indicate adjunction with respect to the inner products defined in \eqref{scalar-products} below. On the right hand side of \eqref{H-op} the factors $(T_{-l}+I)/2$ and $(T_k+I)/2$ take care of projecting back to $\cV$. Their necessity and role is a consequence of the spatial discreteness of $\ZZ^d$. 

One can easily check that 
\begin{align}
\label{sym-id}
&
\text{for any } 
v\in \cV: 
&&
\sum_{k\in{\cN}} 
s_k v_k
=
\nabla^* R^2 v, 
\\
\label{skew-sym-id}
&
\text{for any } 
u\in \cU: 
&&
\sum_{k\in{\cN}} 
b_k u_k
=
\nabla^* H u.
\end{align}
The identity \eqref{sym-id} is straightforward. We check \eqref{skew-sym-id}:
\begin{align*}
\big(\nabla^* H u\big)(\omega)
&
=
\frac12
\sum_{k,l\in\cN} h_{k,l}(\omega) \big(u_l(\tau_k\omega)+u_l(\omega)\big)
\\
&
=
\frac14
\sum_{k,l\in\cN} h_{k,l}(\omega) 
\big(u_l(\tau_k\omega)+u_l(\omega)-u_k(\tau_l\omega)-u_k(\omega)\big)
\\
&
=
\frac14
\sum_{k,l\in\cN} h_{k,l}(\omega) 
\big(2u_l(\omega)-2u_k(\omega)\big)
=
\sum_{k,l\in\cN} h_{k,l}(\omega) u_l(\omega)
=
\sum_{l\in\cN} b_{l}(\omega) u_l(\omega)
\end{align*}
where we have used, in turn, the definition \eqref{H-op} of the operator $H$, the skew-symmetry \eqref{h-tensor} of the tensor $h_{k,l}(\omega)$, \emph{the gradient property of $u\in\cU$}, and the skew-symmetry \eqref{h-tensor} again. Note that the identity \eqref{skew-sym-id} holds \emph{only for} $u\in \cU$ and not for $v\in\cV\setminus\cU$. 

To add to the confusion we note here that the spaces 
\begin{align*}
\cK
&
:=
R\cU 
\hskip14pt
= 
\{(r_k u_k)_{k\in\cN}: u\in\cU\}
\hskip9pt
\subset 
\cV
\end{align*}
will also be relevant in the forthcoming arguments. 

Using \eqref{sym-id} and \eqref{skew-sym-id} 
the Hermitian and anti-Hermitian parts of the infinitesimal generator $L$, defined in \eqref{S-A-ops} are written as
\begin{align}
\label{S-A-op-alt}
&
S
=
\nabla^* R^2 \nabla
=
(\nabla^* R) ( R \nabla)
&&
A
=
\nabla^* H \nabla
=
(\nabla^* R) (R^{-1} H R^{-1}) ( R \nabla)
\end{align}

We will work in the following (Riesz-)Lebesgue spaces 
\begin{align*}
\cL_p
&
:=
\{f\in\cL: 
\norm{f}_p^p
:= 
\int_{\Omega}\abs{f}^p\, d\pi<\infty, 
\ \ 
\int_\Omega f\, d\pi=0\}
\\
\cV_p
&
:=
\{u\in\cV:  \ \ 
\norm{u}_p^p:= 
\frac12\sum_{k\in{\cN}} \norm{u_k}_p^p<\infty\} 
\\
\cU_p
&
:=
\{u\in\cU:  \ \ 
\norm{u}_p^p<\infty,  \ \ 
\int_\Omega u\, d\pi=0\} 
\end{align*}
with $p\in\{1,2,\infty\}$, and the usual interpretation for $p=\infty$.  Note, that only centred (zero mean) functions are kept in the spaces $\cL_p$ and $\cU_p$. However, in the space $\cV_p$ this is not the case.  For $(p,q)\in\{(1, \infty), (2,2), (\infty,1)\}$,  $f\in\cL_p$ and $g\in\cL_q$, respectively,  $u\in\cV_p$ and $v\in\cV_q$,  we define the scalar products
\begin{align}
\label{scalar-products}
&
\sprod{f}{g}
:=
\int_\Omega f(\omega)\, g(\omega)\, d\pi(\omega), 
&&
\sprod{u}{v}
:=
\frac12 \sum_{k\in{\cN}} \sprod{u_k}{v_k}. 
\end{align}
We don't introduce  different notation for the norms and scalar products in $\cL_p$, respectively, $\cV_p$. The precise meaning of $\norm{\cdot}_p$ and $\sprod{\cdot}{\cdot}$ will be always clear from the context. 

Finally, let
\begin{align}
\cK_2
:=
(R \cU_\infty)^{{\tt cl}2}
\buildrel 
{\red \eqref{r-l2}} 
\over 
\subset
\cV_2
\label{K2def}
\end{align}
where the superscript $^{{\tt cl}2}$ denotes closure in $(\cV_2, \norm{\cdot}_2)$. We  will denote by $\Pi:\cV_2\to\cK_2$ the orthogonal projection from $\cV_2$ to $\cK_2$, see \eqref{ortproj}. 

The operators $\partial_k:\cL_2\to\cL_2$, $\nabla:\cL_2\to\cV_2$, $\nabla^*:\cV_2\to\cL_2$ are bounded, and their adjointness relations (with respect to the scalar products \eqref{scalar-products}) are  obviously
\begin{align*}
&
\partial_k^{*}=\partial_{-k}
&& 
(\nabla)^*=\nabla^*.
\end{align*}
The operators $R_k, H_{k,l}$ and $R, H$ defined in \eqref{basic-ops}, \eqref{R-op}, and \eqref{H-op}, when restricted to $\cL_2$, respectively, to $\cV_2$, are \emph{unbounded} with respect to the norms $\norm{\cdot}_2$. However, as multiplication operators there is no issue with their proper definition as densely defined self-adjoint, respectively, skew-self-adjoint operators: 
\begin{align*}
&
R_k=R_k^*, 
&&
H_{k,l}^*=H_{k,l}
&&
R=R^*
&&
H^{*}=-H.
\end{align*}

\subsection{The infinitesimal generator - acting on $\cL_2$}
\label{ss: The infinitesimal generator}

The Hermitian and anti-Hermitian parts (w.r.t. the stationary measure $\pi$) of the infinitesimal generator $L=-S+A$, given in \eqref{S-A-ops} (or, equivalently, in \eqref{S-A-op-alt}) are well defined on the whole space $\cL$ of measurable functions. However, their status as (unbounded) linear operators acting on the Hilbert space $\cL_2$ needs clarification. Actually, they are properly defined as (unbounded) self-adjoint, respectively, skew-self-adjoint operators, and moreover, they share a common core of definition. 
For $K<\infty$,  let 
\begin{align}
\label{Omega_K}
\Omega_K
:=
\{
\omega\in\Omega: 
\max_{k,l,m\in\cN}
\{
r_{k}(\omega)^{\pm1}, r_{k}(\tau_m\omega)^{\pm1}, \abs{h_{k,l}(\omega)}, \abs{h_{k,l}(\tau_m\omega)}
\}
\le  K
\}
\end{align}
and 
\begin{align}
\label{common core}
\cA:=
\{f\in\cL_\infty: \
\exists K<\infty: \
\text{supp}\,f \subset \Omega_K\ \ \ \pi\text{-a.s.}\}.
\end{align}
Obviously, $K\mapsto \Omega_{K}$ exhausts increasingly $\Omega$, as $K\to\infty$. Hence, for $p\in\{1,2\}$, $\cA\subset \cL_\infty \subset \cL_p = \cA^{{\tt cl} p}$. It is also straightforward that $S, A: \cA\to \cL_\infty$, and, for $f,g \in \cA$, 
\begin{align*}
&
\sprod{f}{Sg}=\sprod{Sf}{g}, 
&&
\sprod{f}{Sf}> 0,
&&
\sprod{f}{Ag}=-\sprod{Af}{g}. 
\end{align*}

The proof of the following statement is routine which we omit. 

\begin{proposition}
[The infinitesimal generator]
\label{prop: S is sa A is asa}
\phantom{}
\\
Assume 
{\yellow \eqref{weakell}} 
(and no more). 
The linear operators $S$ and $A$ in \eqref{S-A-ops} (or, equivalently, in \eqref{S-A-op-alt})  acting on $\cL_2$ are essentially self-adjoint, respectively,  essentially skew-self-adjoint on their common core $\cA$ defined in  \eqref{common core}. Moreover, $S>0$, as a self-adjoint operator.
\end{proposition}

\noindent 
{\bf Remark.}
For the purpose of Proposition \ref{prop: S is sa A is asa}, defining the common core $\cA$ in \eqref{common core}, with the simpler and seemingly more natural choice 
\begin{align*}
\Omega_K
:=
\{
\omega\in\Omega: 
\max_{k,m\in\cN}
\{
r_{k}(\omega), r_{k}(\tau_m\omega)
\}
\le  K
\}
\end{align*}
would suffice. However, later, in section \ref{s: Proof of Theorem {main}} we will need the more restrictive choice \eqref{Omega_K}.

\subsection{$\cH_{+}$, $\cH_{-}$, Riesz operators, isometries}
\label{ss: "Sobolev spaces"}

We collect the basic functional analytic facts about the $\cH_{\pm}$ spaces and Riesz operators related to the self-adjoint part $S$ of the infinitesimal generator. In subsection \ref{sss: Abstract} we collect \emph{textbook material} about the  $\cH_{\pm}$ spaces of any positive operator $S$ over a Hilbert space $\cH$, while in subsection \ref{sss: Concrete} we collect those facts which are valid in our specific setting. 

\subsubsection{Abstract setting}
\label{sss: Abstract}

Let $S=S^*>0$ (note the strict inequality!) be a (possibly unbounded) positive operator acting on an  (abstract) separable Hilbert space $(\cH, \norm{\cdot})$. (In our concrete setting $\cH=\cL_2$ and $S=-(L+L^*)/2$ is the self-adjoint part of the infinitesimal generator $L$, given explicitly in \eqref{S-A-op-alt}.)

Since $S=S^*>0$, the self-adjoint and positive (unbounded) operators $S^{-1}$, $S^{1/2}$ and $S^{-1/2}$ are defined through the Spectral Theorem.  Note that 
$\Dom(S)=\Ran(S^{-1})$, 
$\Ran(S)=\Dom(S^{-1})$, 
$\Dom(S^{1/2})=\Ran(S^{-1/2})$, 
$\Ran(S^{1/2})=\Dom(S^{-1/2})$
are dense subspaces in $(\cH, \norm{\cdot})$ and define the Hilbert spaces
\begin{align}
\label{H+ def}
\cH_{+}
&
:=
\big\{
f\in\Dom(S): \
\norm{f}_+^2
:=
\sprod{f}{S f}
\big\}^{{\tt cl}+}
\\[5pt]
\label{H- def}
\cH_{-}
&
:=
\big\{
f\in\Dom(S^{-1}): \
\norm{f}_-^2
:=
\sprod{f}{S^{-1} f}
\big\}^{{\tt cl}-}
\end{align}
The superscripts $^{{\tt cl}\pm}$ denote closure with respect to the Euclidean norm $\norm{\cdot}_{\pm}$, respectively. Thus, $(\cH_{+}, \norm{\cdot}_+)$ and $(\cH_{-}, \norm{\cdot}_-)$ are complete Hilbert spaces. Obviously, if $\norm{S}_{\cH\to\cH}<\infty$ then $\cH_{-}\subset\cH\subset\cH_{+}$, and likewise, if $\norm{S^{-1}}_{\cH\to\cH}<\infty$ then $\cH_{+}\subset\cH\subset\cH_{-}$. However, in most relevant cases $\norm{S}_{\cH\to\cH}=\infty=\norm{S^{-1}}_{\cH\to\cH}$ and neither of the above  subspace-inclusions hold.

Proposition \ref{prop: Hpm general} holds true in full generality, in this abstract setting. These are standard facts available in any textbook on functional analysis, proven with the use of the Spectral Theorem. See, e.g., chapter VIII in \cite{read-simon-vol1}, and/or section 2.2. in \cite{komorowski-landim-olla-12}. The proof of this proposition is textbook material which we omit. The (unbounded) positive operators $S^{1/2}$, $S^{-1/2}$ are defined in terms of the Spectral Theorem. 

\begin{proposition}
[$\cH_{\pm}$ spaces, general abstract setting]
\label{prop: Hpm general} 
\phantom{}

\begin{enumerate} [(i)]

\item 

\begin{align*}
&
\cH_{+}\cap\cH= \Dom(S^{1/2})
\text{ \ \ and for \ }
f\in\cH_{+}\cap\cH, \ \ 
\norm{f}_+=\norm{S^{1/2} f},
\\[5pt]
& 
\cH_{-}\cap\cH=\Ran(S^{1/2})
\text{ \ \ \ and for \ }
f\in\cH_{-}\cap\cH, \ \ 
\norm{f}_-=\norm{S^{-1/2} f}.
\end{align*}

\item 
The following variational formulae hold: for $f\in\cH$
\begin{align}
\notag
&
\norm{f}_{+}^2
=
\sup_{g\in \cH_-\cap\cH}
\big(2\sprod{f}{g} -\norm{g}_-^2\big)
=
\sup_{g\in \Ran(S)}
\big(2\sprod{f}{g} -\sprod{g}{S^{-1}g}\big)
=
\sup_{g\in \Dom(S)}
\sprod{2f-g}{S g}
\\[5pt]
\label{varp-}
&
\norm{f}_{-}^2
=
\sup_{g\in \cH_+\cap\cH}
\big(2\sprod{f}{g} -\norm{g}_+^2\big)
=
\sup_{g\in \Dom(S)}
\big(2\sprod{f}{g} -\sprod{g}{Sg}\big)
=
\sup_{g\in \Ran(S)}
\sprod{2f-g}{S^{-1} g}
\end{align}

\item 
On $\cH_{+}\cap\cH$, respectively, on $\cH_{-}\cap\cH$, define the Euclidean norms
\begin{align*}
& 
\altnorm{\cdot}_{+}^{2}:=\norm{\cdot}_{+}^2+\norm{\cdot}^2, 
&&
\altnorm{\cdot}_{-}^{2}:=\norm{\cdot}_{-}^2+\norm{\cdot}^2.
\end{align*}
Then, 
$(\cH_{+}\cap\cH, \altnorm{\cdot}_{+})$ and $(\cH_{-}\cap\cH, \altnorm{\cdot}_{-})$ are (complete) Hilbert spaces, and 
\begin{align*}
(\cH_{+}\cap\cH , \altnorm{\cdot}_+)
\,\buildrel{\displaystyle S^{1/2}}\over {\displaystyle \longrightarrow}\, 
(\cH_{-}\cap\cH , \altnorm{\cdot}_-)
\, \buildrel{\displaystyle S^{-1/2}}\over {\displaystyle \longrightarrow}\, 
(\cH_{+}\cap\cH , \altnorm{\cdot}_+)
\end{align*}
are isometries between them. 

\item 
The following cycle of isometries holds. 
\begin{align}
\label{a priori cycle}
(\cH, \norm{\cdot}) 
\,\buildrel{\displaystyle S^{-1/2}}\over {\displaystyle \longrightarrow}\, 
(\cH_{+} , \norm{\cdot}_+)
\,\buildrel{\displaystyle S^{1/2}}\over {\displaystyle \longrightarrow}\, 
(\cH, \norm{\cdot}) 
\,\buildrel{\displaystyle S^{1/2}}\over {\displaystyle \longrightarrow}\, 
(\cH_{-} , \norm{\cdot}_-)
\,\buildrel{\displaystyle S^{-1/2}}\over {\displaystyle \longrightarrow}\, 
(\cH, \norm{\cdot}) 
\end{align}

\end{enumerate}

\end{proposition}

\subsubsection{Concrete setting}
\label{sss: Concrete}

\medskip
\noindent
In Proposition  \ref{prop: Hpm particular} we collect those facts about the spaces $\cH_{\pm}$ which are specific to our particular setting, rely on the concrete realization $\cH=\cL_2$ of the basic Hilbert space, and on the special form \eqref{S-A-op-alt} of the operator $S=S^*>0$ in our setting. 

\begin{proposition}
[$\cH_{\pm}$ spaces, concrete setting]
\label{prop: Hpm particular}
\phantom{}
\\
Assume 
{\red \eqref{r-l2}} 
(and nothing more). 

\begin{enumerate} [(i)]

\item 
$\cL_\infty\subset\cH_{+}=\cL_\infty^{{\tt cl}+}$, 
and for $f\in\cL_{\infty}$, 
\begin{align}
\label{norm+}
\norm{f}_+^2=\norm{R\nabla f}_2^2. 
\end{align}

\item 
$\cH_{-}\subset\cL_{1}=\cH_{-}^{{\tt cl}1}$,
and for $f\in\cL_{1}$,  
\begin{align}
\label{norm- var}
\norm{f}_-^2
=
\sup_{g\in\cL_\infty} 
\big(2 \sprod{f}{g} - \norm{R\nabla g}_2^2\big)\leq \infty, 
\end{align}
with the interpretation of $\norm{f}_-=\infty$ as $f\in\cL_1\setminus \cH_-$.

\item 
$\nabla^* R \, \cV_{2} 
=
\nabla^* R \, \cK_{2} 
\subset 
\cH_{-}
=
(\nabla^* R \, \cK_{2})^{{\tt cl}-}$, 
and for $v\in\cV_2$, 
\begin{align}
\label{norm-}
\norm{\nabla^*R v}_{-}^2=\norm{\Pi v}_2^2.
\end{align}

\item 
The following cycle of Hilbert space isometries hold
\begin{align}
\label{my cycle}
(\cL_2, \norm{\cdot}_2) 
\,\buildrel{\displaystyle S^{-1/2}}\over {\displaystyle \longrightarrow}\, 
(\cH_{+} , \norm{\cdot}_+)
\,\buildrel{\displaystyle R\nabla}\over {\displaystyle \longrightarrow}\, 
(\cK_2, \norm{\cdot}_2) 
\,\buildrel{\displaystyle \nabla^*R}\over {\displaystyle \longrightarrow}\, 
(\cH_{-} , \norm{\cdot}_-)
\,\buildrel{\displaystyle S^{-1/2}}\over {\displaystyle \longrightarrow}\, 
(\cL_2, \norm{\cdot}_2) 
\end{align}

\end{enumerate}

\end{proposition}

\medskip
\noindent
{\bf Remarks:}
(1)
Note the difference between the middle terms in \eqref{a priori cycle} and \eqref{my cycle}. The main point is that the abstract and \emph{non-computable} operator $S^{1/2}$ is substituted by the \emph{explicitly computable} $R\nabla$ and $\nabla^* R$, and, accordingly, in the middle of the chain the Hilbert space $(\cL_2, \norm{\cdot}_2)$ is replaced by $(\cK_2, \norm{\cdot}_2)$. 
\\[5pt]
(2)
$\nabla^*R: (\cV_2, \norm{\cdot}_2) \to (\cH_{-}, \norm{\cdot}_{-})$
is actually a \emph{partial} isometry with $\Ker(\nabla^*R)=\cK_2^{\perp}$. We will denote
\begin{align}
\label{Riesz ops}
&
\Lambda :=
R\nabla S^{-1/2}
: \cL_2\to\cK_2,
&&
\Lambda^* :=
S^{-1/2} \nabla^* R
: \cV_2\to\cL_2, 
\end{align}
and refer to them as the \emph{Riesz operators}, due to their apparent  analogy to the Riesz operators (${\tt grad}\abs{\Delta}^{-1/2}$, respectively, $\abs{\Delta}^{-1/2}{\tt div}\cdot$) of harmonic analysis. We obviously have
\begin{align}
\label{ortproj}
&
\Lambda^*\Lambda=I_{\cL_2}, 
&&
\Pi :=
\Lambda\,\Lambda^*
:\cV_2\to \cK_2, 
\end{align}
the latter one being the orthogonal projection from $\cV_2$ to $\cK_2$. 
 
\begin{proof}
[Proof of Proposition \ref{prop: Hpm particular}]
\phantom{.}

\medskip
\noindent
(i)
Due to 
{\red \eqref{r-l2}},  
for $f\in\cL_\infty$ we have $Sf=\nabla^* RR\nabla f\in\cL_1$, $R\nabla f\in\cK_2$, and the following identities are  legitimate:
\begin{align*}
\norm{f}_+^2
=
\sprod{f}{\nabla^* RR\nabla f}
=
\sprod{R\nabla f}{R\nabla f}
=
\norm{R\nabla f}_2^2
\end{align*}
It follows that $\cL_\infty\subset \cH_+\cap\cL_2=\Dom(S^{1/2})$ and \eqref{norm+} holds.  

In order to prove $\cH_+=\cL_\infty^{{\tt cl}+}$ note that, for $f\in \cL_\infty$ and $g\in\Dom(S^{1/2})$, $\sprod{f}{g}_+=\sprod{S^{1/2}f}{S^{1/2}g}$. Hence (since $\cL_2=\Ran(S^{1/2})^{{\tt cl}2}$) $\sprod{f}{g}_+=0$ for all $g\in\Dom(S^{1/2})$ implies $S^{1/2}f=0$, and thus $f=0$. 

\medskip
\noindent
(ii)
From \eqref{varp-}, $\cL_\infty\subset\cH_+\cap\cH$ and \eqref{norm+} we get, for $f\in\cL_1$
\begin{align*}
\norm{f}_-^2
&
\geq
\sup_{g\in\cL_\infty} \big(2\sprod{f}{g}- \norm {R\nabla g}_2^2\big)
\\
&
\geq
\sup_{g\in\cL_\infty} \big(2\sprod{f}{g}- 4 \sum_k\norm{r_k}_2^2 \norm {g}_\infty^2\big)
=
\big(4 \sum_k\norm{r_k}_2^2 \big)^{-1}
\norm{f}_1^2
\end{align*}
and hence, due to 
{\red \eqref{r-l2}}, 
$\cH_-\subset \cL_1$.
$\cL_{1}=\cH_{-}^{{\tt cl}1}$ follows from $\cL_\infty\subset\cH_+$ and Hahn-Banach. 

\medskip
\noindent
(iii)
To prove \eqref{norm-}
let $v\in \cV_2$ and apply \eqref{norm- var} to 
$\nabla^*Rv \buildrel {\red \eqref{r-l2}} \over \in\cL_1$:
\begin{align*}
\norm{\nabla^*Rv}_-^2
=
\sup_{g\in\cL_\infty}
\big(2 \sprod{\nabla^*Rv}{g} - \norm{R\nabla g}_2^2\big)
=
\sup_{g\in\cL_\infty}
\big(2 \sprod{v}{R\nabla g} - \norm{R\nabla g}_2^2\big)
=\norm{\Pi v}_2^2.
\end{align*}
Above, the middle step is legitimate since 
$g\in\cL_\infty$, 
$v\in\cK_2$, 
$\nabla^*Rv \buildrel {\red \eqref{r-l2}} \over \in\cL_1$, and
$R\nabla g  \buildrel {\red \eqref{r-l2}} \over \in\cL_2$. 
The last step holds 
since $\cK_2 = (R\nabla \cL_\infty)^{{\tt cl}2}$, by definition \eqref{K2def}.  
\\
To prove 
$\cH_{-}
=
(\nabla^* R \, \cK_{2})^{{\tt cl}-}$
let $g\in\cL_\infty$ be such that for all $v\in \cK_2$
\begin{align*}
0
=
\sprod{\nabla^* R v}{g} 
=
\sprod{v} {R\nabla g}.
\end{align*}
Then $g=0$ and hence, since $\cH_+=\cL_\infty^{{\tt cl}+}$, it follows that $\cH_{-}=(\nabla^* R \, \cK_{2})^{{\tt cl}-}$

\medskip
\noindent
(iv)
We only have to prove the middle terms in \eqref{my cycle}. 
From Proposition \ref{prop: Hpm particular} (i) readily follows that $R\nabla$ maps isometrically $(\cH_+, \norm{\cdot}_+)$ \emph{into} $(\cK_2, \norm{\cdot}_2)$. To prove that it is also \emph{surjective} assume that  $v\in\cK_2$ is such that for all $g\in\cL_\infty$
\begin{align*}
0
=
\sprod{v}{R\nabla g}
=
\sprod{\nabla^* R v}{g}, 
\end{align*}
and thus, due to $\nabla^* R \cK_2 \subset \cH_- \subset \cL_1$, we get $\nabla^* R v =0$, and  hence, $v=0$ due to \eqref{norm-}. 

\end{proof}

\section{Diffusive bounds}
\label{s: Diffusive bounds}

In this section we prove annealed diffusive bounds for the displacement of the random walker.

\subsection{Upper}
\label{ss: Upper}

\begin{lemma}
[$\cH_-$-bounds on drifts]
\label{lem: H- bounds}
\phantom{}
\\
Assume 
{\red \eqref{r-l2}}
and 
{\green \eqref{new H-1}}. 
Then, 
\begin{align}
\label{H- bounds}
\phi_i, \psi_i \in\cH_{-}, 
\qquad
i=1,\dots,d, 
\end{align}
where $\phi_i$ and  $\psi_i$ are the components of the drift-fields defined in \eqref{drift fields}.  
\end{lemma}

\begin{proof}
We rely on \eqref{norm- var} from Proposition \ref{prop: Hpm particular} (ii), . 
\begin{align*}
\norm{ \partial_{-k} s_k}_{-}^2
& 
=
\sup_{\chi\in\cL_\infty}
\big(
2\sprod{\chi}{\partial_{-k} s_k} 
- 
\norm{R\nabla  \chi}_2^2
\big)
\leq
\sup_{\chi\in\cL_\infty}
\big(
2\sprod{\chi}{\partial_{-k} s_k} 
- 
\norm{R_k\partial_k \chi}_2^2
\big)
\\[5pt]
&
=
\sup_{\chi\in\cL_\infty}
\big(
2 \sprod{R_k \partial_{k}\chi}{r_k}
- 
\norm{R_k\partial_k \chi}_2^2
\big)
\leq
\sup_{\chi\in\cL_2}
\big(
2\sprod{\chi}{r_k}
- 
\norm{\chi}_2^2\big)
=
\norm{r_k}_2^2
\buildrel
{\red \eqref{r-l2}} 
\over<
\infty
\end{align*}
\begin{align*}
\norm{b_k}_-^2
&
= 
\sup_{\chi\in\cL_\infty}
\big(
\sprod{\chi}{ \sum_{l\in{\cN}} \partial_{-l}h_{k,l}} 
- 
\norm{R \nabla\chi}_2^2 
\big)
=
\sup_{\chi\in\cL_\infty}
\sum_{l\in{\cN}}
\big(
\sprod{R_l \partial_l \chi}{r_l^{-1} h_{k,l}}
- 
\norm{R_l \partial_l\chi}_2^2 
\big)
\\[5pt]
&
\leq
\sum_{l\in{\cN}}
\sup_{\chi\in\cL_2}
\big(
\sprod{\chi}{r_l^{-1} h_{k,l}}
- 
\norm{\chi}_2^2
\big)
=
\sum_{l\in{\cN}}
\norm{r_l^{-1}h_{k,l}}_2^2
\buildrel
{\green \eqref{new H-1}} 
\over<
\infty.
\end{align*}
Due to \eqref{drift fields}, from the previous computations we readily get 
\begin{align*}
& 
\sum_{i=1}^d
\norm{\phi_i}_-^2 
\leq 
\sum_{k\in{\cN}} 
\norm{r_k}_2^2
\buildrel
{\red \eqref{r-l2}} 
\over<
\infty,
&& 
\sum_{i=1}^d
\norm{\psi_i}_-^2 
\leq 
\sum_{k,l\in{\cN}} 
\norm{r_l^{-1} h_{k,l}}_2^2
\buildrel
{\green \eqref{new H-1}} 
\over<
\infty,
\end{align*}
which proves the claims in \eqref{H- bounds}.
\end{proof}

\begin{proposition}
[Diffusive upper bound]
\label{prop:  upper}
\phantom{}
\\
Assume 
{\red \eqref{r-l2}}
and 
{\green \eqref{new H-1}}. 
Then, 
\begin{align}
\label{upper-bound}
\limsup_{t\to\infty} t^{-1} \expect{\abs{X(t)}^2} 
<
\infty
\end{align}
\end{proposition}

\begin{proof}
We use martingale decomposition \eqref{mart-decomp} and provide diffusive upper bounds  separately  for the three terms on the right hand side. 

The first term $t\mapsto M(t)\in\RR^d$ is a martingale whose conditional covariance (bracket) process is
\begin{align*}
\langle M_i,M_j\rangle(t)
=
\delta_{i.j} \left(p_{e_i}(\eta_s)+p_{-e_i}(\eta_s)\right)
\end{align*}
Hence, in the annealed setting, 
\begin{align}
t^{-1}\expect{\abs{M(t)}^2}
=
\sum_{k\in{\cN}}\int_{\Omega}  s_{k}(\omega) \, d\pi(\omega)
=
\sum_{k\in{\cN}}\int_{\Omega}  r_{k}^2(\omega) \, d\pi(\omega)
\buildrel
{\red \eqref{r-l2}} 
\over<
\infty
\label{M-bound}
\end{align}

On the other hand, recalling   Varadhan's $\cH_{-}$-bound (cf. \cite{komorowski-landim-olla-12}, Theorem 2.7) from \eqref{H- bounds} (which relies on 
{\red \eqref{r-l2}} 
and 
{\green \eqref{new H-1}})
we conclude
\begin{align}
\label{IJ-bound}
\sup_{0<t<\infty}
t^{-1}
\big(
\expect{\abs{I(t)}^2}
+
\expect{\abs{J(t)}^2}
\big)
<\infty.
\end{align}
In the end, \eqref{upper-bound} follows from \eqref{M-bound} and \eqref{IJ-bound}. 

\end{proof}

\noindent
{\bf Remark:}
In \cite{komorowski-landim-olla-12}, Theorem 2.7, the upper bound is stated for $f\in\cH_{-}\cap \cL_2$. However, in fact, the proof works equally well for $f\in\cH_{-}\cap \cL_1$, which covers our assumption 
{\red \eqref{r-l2}}.

\subsection{Lower}
\label{ss: Lower}

The following proposition is adapted from the unpublished note \cite{toth-86b}. 

\begin{proposition}
[Diffusive lower bound]
\label{prop:lower}
\phantom{}
\\
Assume 
{\blue \eqref{rrec-l2}}. 
For any unit vector $e\in\RR^d$
\begin{align*}
\liminf_{t\to\infty} t^{-1} \expect{(e\cdot X(t))^2} 
\geq 
\sum_{k\in\cN}
\ol{s}_k (e\cdot k)^2
\end{align*}
whith $\ol{s}_k>0$, $k\in\cN$, defined in \eqref{s-bar-def} below. 
\end{proposition}

\begin{proof}
We must start with some notation. Let 
\begin{align}
\ol{s}_k:= 
\left(\int _\Omega s_k(\omega)^{-1} \, d\pi (\omega)\right)^{-1}
\buildrel
{{\blue {\eqref{rrec-l2}}}}
\over
> 0
\label{s-bar-def}
\end{align}
and note that, due to \eqref{s-cond},  $\ol{s}_k=\ol{s}_{-k}$. 
Further, let $\alpha, \beta :\Omega \to \RR^d$, 
\begin{align*}
\alpha(\omega)
&
:=
\sum_{k\in\cN}
p_k(\omega) \ol{s}_k s_k(\omega)^{-1}k
=
\sum_{k\in\cN}
\ol{s}_k b_k(\omega) s_k(\omega)^{-1}k
\\[5pt]
\beta(\omega)
&
:=
\phi(\omega)+\psi(\omega) - \alpha(\omega)
\end{align*}
(see \eqref{phi-and-psi} for the definition of $\phi, \psi:\Omega\to \RR^d$).

We denote
\begin{align*}
\theta_0
&
:=
0, 
\qquad
\theta_{i+1}
:=
\inf\{t>\theta_i: \lim_{\vareps\searrow0}\left(X(t+\varepsilon)-X(t-\varepsilon)\right)\not=0\} 
\\[5pt]
\xi_{i}
&
:=
\lim_{\vareps\searrow0}\left(X(\theta_i+\varepsilon)-X(\theta_i-\varepsilon)\right)\not=0\} \in\cN
\\[5pt]
\eta_i^-
&
:=
\eta_{\theta_i^-}, 
\qquad 
\eta_i^+
:=
\eta_{\theta_i^+}.
\end{align*}
In plain words, $\theta_i$, $i\geq1$, is the time of the $i$-th jump of the continuous time nearest neighbour walk walk $t\mapsto X(t)$, $\xi_i\in \cN$ is the $i$-th jump and $\eta_i^-$ / $\eta_i^+$  is the environment seen right before/after the $i$-th jump.   
Finally, we define the c.a.d.l.a.g. processes $t\mapsto Z(t)\in\RR^d$ and  $t\mapsto Y(t)\in\RR^d$
\begin{align*}
Z(t)
&
:=
\sum_{i: 0<\theta_i<t}
\ol{s}_{\xi_i} s_{\xi_i}(\eta_i^-)^{-1} \xi_i
-\int_0^t \alpha(\eta_s)\, ds
\\[5pt]
Y(t)
&
:=
\sum_{i: 0<\theta_i<t}
\left( 1 -\ol{s}_{\xi_i} s_{\xi_i}(\eta_i^-)^{-1} \right) \xi_i
-\int_0^t \beta(\eta_s)\, ds
\end{align*}
and note that these are \emph{quenched martingales} adapted to the natural \emph{forward filtration} $\cF_t:=\sigma(\eta_s: s\leq t)$. Indeed, the "compensators" $\alpha$ and $\beta$ are chosen so that $\pi$-a.s.
\begin{align}
\label{fwmart}
\lim_{h\searrow0}h^{-1}
\condexpectom{Z(t+h)-Z(t)}{\cF_t}
=
\lim_{h\searrow0}h^{-1}
\condexpectom{Y(t+h)-Y(t)}{\cF_t}
=
0
\end{align}
readily follows. 

Moreover, $t\mapsto Z(t)$ is also a \emph{backward martingale}, with respect to the \emph{backward filtration} $\cF^*_t:=\sigma(\eta_s: s\geq t)$ and \emph{under the annealed/stationary measure}. Compute 
\begin{align}
\notag
\condexpect{Z(t-h)-Z(t)}{\cF^*_t}
&
=
\condexpect{\sum_{i: t-h<\theta_i<t}
\ol{s}_{\xi_i} s_{\xi_i}(\eta_i^-)^{-1} (-\xi_i)
}{\cF^*_t}
+\alpha(\eta_t) h +\ordo(h)
\\
\notag
&
=
\condexpect{\sum_{i: t-h<\theta_i<t}
\ol{s}_{-\xi_i} s_{-\xi_i}(\eta_i^+)^{-1} (-\xi_i)
}{\cF^*_t}
+\alpha(\eta_t) h +\ordo(h)
\\
\label{bwmart-prel}
&
=
\condexpect{\sum_{i: t<\theta^*_i<t+h}
\ol{s}_{\xi^*_i} s_{\xi^*_i}(\eta_i^{*-})^{-1} \xi_i^*
}{\cF^*_t}
+\alpha(\eta^*_t) h +\ordo(h)
\end{align}
where $t\mapsto \eta^*_t:=\eta_{-t}$ is the stationary time-reversed environment process, and $\theta^*_i$, $\xi^*_i$, $\eta_i^{*\pm}$ denote the jump times, jumps and environment seen right before/after jumps along the \emph{time-reversed trajectory}. In the above computation the first step is just explicit expression of the left hand side, in the secind step we use the symmetries $\ol{s}_k=\ol{s}_{-k}$ and $s_k(\omega)=s_{-k(\tau_k\omega)}$. Finally, in the third step we switch from the forward to the backward dynamics. Note the substantial difference between \eqref{fwmart} and \eqref{bwmart-prel}: while in \eqref{fwmart} the conditional expectations are taken under the \emph{quenched} measure $\condexpectom{\cdot}{\cF_t}$, in  \eqref{bwmart-prel} these are taken under the annealed (and stationary) measure $\condexpect{\cdot}{\cF_t^*}$.

The last ingredient is the observation that the stationary (annealed) time-reversed environment process $t\mapsto \eta^*_t$ is similar to the forward process $t\mapsto \eta_t$, with jump rates
\begin{align*}
&
p^*_k(\omega) 
= 
p_{-k}(\tau_k\omega),
&&
\text{that is: }
&&
s^*_k(\omega)
=
s_k(\omega), 
&&
b^*_k(\omega)
=
-b_k(\omega).  
\end{align*}
Hence, 
\begin{align*}
\condexpect{\sum_{i: t<\theta^*_i<t+h}
\ol{s}_{\xi^*_i} s_{\xi^*_i}(\eta_i^{*-})^{-1} \xi_i^*
}{\cF^*_t}
=
h 
\sum_{k\in\cN}
p^*_k( \eta^*_t)
\ol{s}_k s_k( \eta^*_t)^{-1}k
+\ordo(h)
=
-\alpha( \eta^*_t) h +\ordo(h).
\end{align*}
Putting all these together we get indeed
\begin{align}
\label{bwmart}
\lim_{h\searrow0}h^{-1}
\condexpectom{Z(t-h)-Z(t)}{\cF^*_t}
=
0. 
\end{align}

Note, however, the subtle difference between \eqref{fwmart} and \eqref{bwmart}. The conditional expectations in \eqref{fwmart} are meant under the \emph{quenched} probability measure  $\probabom{\dots}$ and for $\pi$-almost all $\omega\in\Omega$ while the conditional expectation in \eqref{bwmart} is taken under the annealed/stationary measure.

Thus, we have the decomposition \eqref{X-decomp-alt} (to be compared with \eqref{mart-decomp}):
\begin{align}
\label{X-decomp-alt}
X(t)
=
\underbrace{Z(t)+Y(t)}_{\displaystyle M(t)} + 
\underbrace{\int_{0}^{t} \left(\phi(\eta_s)+\psi(\eta_s)\right)\, ds}_{\displaystyle I(t)+J(t)}. 
\end{align}

Straightforward computations yield the conditional covariance (a.k.a. bracket-) processes
\begin{align*}
\condexpectom {dZ(t) \land dZ(t)}{\cF_t}
& 
=
\Phi(\eta_t)\, dt
\\[5pt]
\condexpectom {dZ(t) \land dY(t)}{\cF_t}
& 
=
\Psi(\eta_t) \, dt
\end{align*}
where 
$\Phi, \Psi: \Omega\to \RR^{d\times d}$ are 
\begin{align*}
\Phi(\omega)
&
:=
\sum_{k\in\cN}
\ol{s}_k
\left(\ol{s}_k s_k(\omega)^{-1} + \ol{s}_k b_{k}(\omega) s_k(\omega)^{-2}\right) k\land k
\\[5pt]
\Psi(\omega)
&
:=
\sum_{k\in\cN}
\ol{s}_k
\left(1 + b_{k}(\omega) s_k(\omega)^{-1}\right) 
\left(1- \ol{s}_k s_k(\omega)^{-1}\right)
k\land k.
\end{align*}
(Since the bracket process $\condexpect{dY(t) \land dY(t)}{\cF_t}$ is not needed in the forthcoming argument we don't write down its explicit expression.)

Next, using the (anti)symmetries in \eqref{s-cond} and \eqref{b-flow} we  readily obtain 
\begin{align}
\label{Phi-Psi-int}
&
\int_{\Omega} 
\Phi(\omega) \, d\pi(\omega)
=
\sum_{k\in\cN} 
\ol{s}_k k\land k,
&&
\int_{\Omega} 
\Psi(\omega) \, d\pi(\omega)
=
0
\end{align}
Summarizing: From $t\mapsto Z(t)$ and $t\mapsto Y(t)$ being (forward) martingales and the second identity in \eqref{Phi-Psi-int} it follows that for any $t_1, t_2 >0$
\begin{align}
\label{uncorr1}
\expect{Z(t_1)Y(t_2)} =0. 
\end{align}
Further on, from $t\mapsto Z(t)$ being a \emph{forward-and-backward} martingale it follows that for any  $t_1, t_2 >0$
\begin{align}
\label{uncorr2}
\expect{Z(t_1)(I(t_2)+J(t_2)} =0. 
\end{align}
Finally from \eqref{X-decomp-alt}, \eqref{uncorr1}, \eqref{uncorr2} we obtain 
\begin{align*}
\expect{X(t)\land X(t)}
\geq 
\expect{Z(t)\land Z(t)}
=
t 
\sum_{k\in\cN} 
\ol{s}_k k\land k
\end{align*}
where the first inequality is meant in the $d\times d$-matrix sense and the last identity follows from \eqref{Phi-Psi-int}. 

\end{proof}

\section{Proof of Theorem \ref{thm: main}}
\label{s: Proof of Theorem {main}}

\subsection{Kipnis-Varadhan theory - the abstract form}
\label{ss: Kipnis-Varadhan}

The proof of Theorem \ref{thm: main} is based on the non-reversible (i.e. non-self-adjoint) and non-graded version of martingale approximation a la Kipnis-Varadhan, whose abstract form is summarized concisely in this section.

Let $(\Omega, \cF, \pi)$ be a probability space and $t\mapsto\eta_t\in\Omega$ a Markov process assumed to be stationary and ergodic under the probability measure $\pi$, whose infinitesimal generator $L$ and resolvent $R_\lambda:= (\lambda I-L)^{-1}$ act on the Lebesgue spaces 
\begin{align*}
\cL_{p}
:=
\{f\in\cL_p(\Omega, \pi): \int_\Omega f \, d\pi=0\}, 
\qquad
p\in\{1,2,\infty\}.
\end{align*}
We assume that the infinitesimal generator acting on $\cL_2$ decomposes as 
\begin{align*}
&
L=-S+A, 
&&
S:=-(L+L^*)/2, 
&&
A:=(L-L^*)/2,  
\end{align*}
meaning that the symmetric and antisymmetric parts, $S$, respectively, $A$, have a common subspace of definition $\cA\subset\cL_2$ which serves as a core for their closure as (unbounded) linear operators acting on $\cL_2$. Moreover, we assume that $S, A:\cA\to \cL_2$ are essentially self-adjoint, respectively, essentially skew-self-adjoint. We assume that the self-adjoint part $S$ is ergodic on its own, that is its eigenvalue $0$ is non-degenerate. 

We also define the self-adjoint operators $S^{1/2}$, $S^{-1/2}$, in terms of the Spectral Theorem, and the Hilbert spaces $(\cH_+, \norm{\cdot}_+)$, $(\cH_-, \norm{\cdot}_-)$ as in \eqref{H+ def}, \eqref{H- def}.  

We quote the Kipnis-Varadhan martingale approximation in the non-self-adjoint setting in the form stated in  \cite{toth-86a,horvath-toth-veto-12, toth-13}. See the monograph \cite{komorowski-landim-olla-12} for historic background. 

\begin{theorem}
[\cite{toth-86a,horvath-toth-veto-12, toth-13} Theorem KV]
\label{thm: toth-86a}
\phantom{}
\\
Let $\varphi\in \cL_1$ such that for all $\lambda >0$, $R_\lambda \varphi\in \cL_2$. If the following two conditions hold
\begin{align}
\label{condition A&B}
& 
\lim_{\lambda\to0} 
\lambda^{1/2}\norm{R_\lambda \varphi}_2=0, 
&&
\lim_{\lambda\to0} 
\norm{S^{1/2}R_\lambda\varphi-v}_2=0, \ \ v\in\cL_2,  
\end{align}
then  there exists a square integrable martingale $t\mapsto Z(t)$, with stationary and
ergodic increments, adapted to the natural filtration $(\cF_t)_{0\leq t<\infty}$ of the Markov process $t\mapsto\eta_t$, and  with variance
\begin{align*}
\expect{\abs{Z(t)}^2}=2 \norm{v}_2^2 \, t,
\end{align*}
such that 
\begin{align}
\label{KV-martingale-approximation}
\lim_{t\to\infty}
t^{-1} \expect{\abs{\int_0^t \varphi(\eta_s)\, ds - Z(t)}^2}=0. 
\end{align}
\end{theorem}

\medskip
\noindent
{\bf Remarks, comments:}
(1) 
In the common traditional formulations of this or equivalent theorems (including, e.g.,  \cite{toth-86a}, \cite{olla-01}, \cite{komorowski-landim-olla-12}, \cite{horvath-toth-veto-12}, \cite{toth-13}, \cite{toth-18b}) it is assumed that $\varphi\in\cL_2$. However, as noted already in \cite{demasi-et-al-89} for the reversible/self-adjoint setting, the arguments work neatly under the weaker assumption $\varphi\in\cL_1$ and $R_{\lambda}\varphi\in\cL_2$, for all $\lambda>0$.
\\[3pt]
(2)
By general abstract arguments, for all $\lambda>0$
\begin{align*}
&
\norm{R_\lambda}_{2\to2}
\leq \lambda^{-1},
&&
\norm{S^{1/2} R_\lambda }_{2\to2}
\leq \lambda^{-1/2},
&&
\norm{R_\lambda S^{1/2}}_{2\to2}
\leq \lambda^{-1/2},
&&
\norm{S^{1/2}R_\lambda S^{1/2}}_{2\to2}
\leq 1,
\end{align*}
and hence, for $\varphi\in \cH_{-}$ we have a priori
\begin{align*}
&
\sup_{\lambda>0}
\lambda^{1/2}\norm{R_\lambda \varphi}_2<\infty, 
&&
\sup_{\lambda>0}
\norm{S^{1/2}R_\lambda \varphi}_2<\infty,
\end{align*}
still somewhat short of \eqref{condition A&B}.
\\[3pt]
(3)
Conditions \eqref{condition A&B} of Theorem \ref{thm: toth-86a} are difficult to check directly. (An exception, where this could be done, is the RWRE problem treated in \cite{toth-86a}.) Sufficient conditions are known under the names of Strong Sector Condition, respectively,  Graded Sector Condition. See the monograph \cite{komorowski-landim-olla-12} for context and details. However, these sector conditions hold only under very special structural assumptions about the Markov process considered: a graded structure of the infinitesimal generator $L$ acting on an accordingly graded Hilbert space $\cL_2$. This structural assumption simply doesn't hold in many cases of interest, including our current problem. 

\bigskip
\noindent
The next theorem, quoted from \cite{horvath-toth-veto-12}, provides a sufficient condition which does not assume a graded structure of the infinitesimal generator $L$ acting on the Hilbert space $\cL_{2}(\Omega,\pi)$. Let
\begin{align}
\label{B core abstract}
\cB:=
\{f\in\cH_-\cap\cL_2: \ 
S^{-1/2}f \in \Dom(A),  \ \  AS^{-1/2}f\in\cH_-\cap\cL_2
\}
\end{align}
and $B:\cB\to\cL_{2}$ defined as
\begin{align}
\label{B op abstract}
Bf
:=
S^{-1/2}A S^{-1/2}f. 
\end{align}
Note that the operator $B:\cB\to\cL_{2}$ is unbounded (except for the cases when the operator $S:\cL_{2}\to \cL_{2}$ is boundedly invertible) and  skew symmetric. Indeed, for $f,g\in\cB$ all the straightforward steps below are legitimate
\begin{align*}
\sprod{f}{S^{-1/2}AS^{-1/2}g}
\!
=
\!
\sprod{S^{-1/2}f}{AS^{-1/2}g}
\!
=
\!
-
\sprod{AS^{-1/2}f}{S^{-1/2}g}
\!
=
\!
-
\sprod{S^{-1/2}AS^{-1/2}f}{g}
\end{align*}
Of course, it could happen that the subspace $\cB$ is not dense in $\cL_{2}$, or, even worse, that simply $\cB=\{0\}$. Even if $\cB$ is a dense subspace  in $\cL_{2}$, in principle it could still happen that the operator $B$ (which in this case is densely defined and skew-symmetric) is not essentially skew-self-adjoint.

\begin{theorem} 
[\cite{horvath-toth-veto-12} Theorem 1]
\label{thm: rsc}
\phantom{}
\\
Assume that there exists a subspace $\cC\subseteq\cB\subset \cL_2 = \cC^{{\tt cl}2}$, and the operator $B:\cC\to\cL_2$ is essentially skew-self-adjoint (that is,  $\overline B=-B^*$).  Then for any $\varphi\in\cH_-\cap \cL_1$ the conditions of Theorem \ref{thm: toth-86a} (and hence the martingale approximation \eqref{KV-martingale-approximation}) hold.
\end{theorem}

\noindent
{\bf Remarks:} 
(1) 
In \cite{horvath-toth-veto-12} the theorem is formulated in slightly different terms. However, it is easy to see that this form follows directly from that of Theorem 1 in \cite{horvath-toth-veto-12}. 
\\[3pt]
(2)
The conditions of Theorem \ref{thm: rsc} are equivalent to $\cB\subset \cL_{2} = \cB^{{\tt cl}2}$ and $B:\cB\to\cL_2$ essentially skew-self-adjoint. The formulation of the theorem allows flexibility in choosing the core $\cC\subseteq \cB$. 
\\[3pt]
(3)
In view of Proposition \ref{prop: Hpm general} (iv) \eqref{a priori cycle}, the requirement that $B:=S^{-1/2}AS^{-1/2}$ be a well-defined skew-self-adjoint operator acting on the Hilbert space $(\cL_2, \norm{\cdot}_2)$ is equivalent to $A=(L-L^*)/2$ be a well-defined skew-self-adjoint operator acting on the Hilbert space $(\cH_+, \norm{\cdot}_+)$.
\\[3pt]
(4)
The condition stated in Proposition 2.1.2 of \cite{olla-01} is precisely a somewhat hidden form of von Neumann's criterion for (skew-)self-adjointness (cf. \cite{read-simon-vol1} Theorem VIII.3) of the operator $B$ defined above. The formulation in Theorem \ref{thm: rsc} allows for checking skew-self-adjointness in various other ways, as demonstrated in the forthcoming proof of our main result. 

\subsection{Proof of Theorem \ref{thm: main}}
\label{ss: Checking RSC}

We check the conditions of Theorem \ref{thm: rsc} for the concrete case under consideration, when the operators $L$, $S=-(L+L^*)/2$, and $A=(L-L^*)/2$ given in \eqref{inf-gen-op}, \eqref{S-A-ops}, \eqref{S-A-op-alt} are the infinitesimal generator of the environment process $t\mapsto\eta_t$ cf. \eqref{envproc}, acting on the Hilbert space $\cL_2$, and its self-adjoint and skew-self-adjoint parts. 

In this case the subspace $\cB$ of \eqref{B core abstract} is 
\begin{align*}
\cB
:=
\{
f
\in\cH_-\cap \cL_2: \ 
S^{-1/2}f \in \Dom(\nabla^* H \nabla), 
\ \ 
\nabla^* H \nabla S^{-1/2}f\in\cH_-\cap \cL_2
\}
\end{align*}
and the operator $B:\cB\to\cL_2$ of \eqref{B op abstract} acts as
\begin{align*}
Bf:= 
S^{-1/2} \nabla^* H \nabla S^{-1/2} f
=
\Lambda^* R^{-1} H R^{-1}  \Lambda f
=
\Lambda^* \Pi R^{-1} H R^{-1}  \Pi \Lambda f
\end{align*}
Recall from \eqref{common core} the common core $\cA$ of the operators $S$ and $A$, and (noting that $\cA\subset \Dom(S) \subset \Dom(S^{1/2})$) let 
\begin{align*}
\cC
&
:=
\{f=S^{1/2} g: g \in \cA\}
\end{align*}
The proof of essential skew-self-adjointness of the operator $B:\cC\to\cL_2$ relies on various parts of Proposition \ref{prop: Hpm particular},  thus only on the integrability condition 
{\red \eqref{r-l2}}. 

Obviously, $\cC\subseteq \cH_{-}\cap \cL_{2}$  and since $\cA$ is a core for $S$ and $\Ker(S)=\Ran(S)^{\perp}=\{0\}$,  we also have  $\cL_2=\cC^{{\tt cl}2}$. For $f\in\cC$, the equation $f=S^{1/2}g$ determines \emph{uniquely} $g\in\cA \subset \cL_\infty$. Furthermore, 
\begin{align*}
\nabla^*H\nabla S^{-1/2} f
=
\nabla^ * R \,
\underbrace{R^{-1} H \nabla 
\underbrace{g}_{\in\cA}}_{\in\cV_\infty}
\in\cH_-\cap\cL_{2}. 
\end{align*}
The last step is due to Proposition \ref{prop: Hpm particular} (iii). 
Thus, indeed, 
\begin{align*}
\cC\subset\cB\subset\cL_2=\cC^{{\tt cl}2}, 
\end{align*}
where the superscript $^{{\tt cl}2}$ denotes closure with respect to the norm $\norm{\cdot}_2$. 

\begin{proposition}
[Skew-self-adjointness of $S^{-1/2}AS^{-1/2}$]
\label{prop: B-is-skew-self-adjoint}
\phantom{}
\\
The linear operator $B:\cC\to\cL_2$ is essentially skew-self-adjoint.  
\end{proposition}

\begin{proof}
Let 
\begin{align*}
&
\cD:=
\Lambda \cC
=
\{R \nabla g: g\in\cA\}
\subset \cK_2
= 
\cD^{{\tt cl}2}
\end{align*}
and define the operator $D:\cD\to \cK_2$ as $D:=
\Lambda B|_{\cC} \Lambda^*$. That is, 
\begin{align}
\label{op-D-def}
&
D:= 
\Pi R^{-1}H R^{-1} \Pi, 
&&
Du
:=
\underbrace{\Pi \underbrace{R^{-1}H\underbrace{R^{-1}\underbrace{u}_{\in\cD}}_{\in\nabla\cA}}_{\in\cV_{\infty}}}_{\in\cK_{2}},
\end{align}
where $\Lambda, \Lambda^*, \Pi$ are the Riesz operators and orthogonal projection defined in \eqref{Riesz ops}, \eqref{ortproj}. 

Since $\Lambda: \cL_2\to \cK_2$, $\Lambda^*:\cK_2\to\cL_2$ are \emph{isometries}, the proposition will follow from proving that the operator $D:\cD\to\cK_2$ is essentially skew-self-adjoint. This is what we are going to prove. 

The operator $D:\cD\to\cK_2$ is skew symmetric. 
Indeed, from the choice of the core $\cA$ is follows that $u\in\cD\subset \cK_\infty$  readily implies that
$R^{-1}u\in \nabla \cA \subset \cU_\infty$, 
$HR^{-1}u \in \cV_\infty$, 
$R^{-1}HR^{-1}u \in \cV_\infty$ also hold, and thus all steps of the following chain are legitimate (not merely formal)
\begin{align*}
\sprod{ u}{ R^{-1}HR^{-1} v}
=
\sprod{ R^{-1} u}{ H R^{-1} v}
=
-
\sprod{ H R^{-1} u}{R^{-1} v}
=
-
\sprod{ R^{-1}HR^{-1} u}{ v}
\end{align*}
Next we define the adjoint (over the Hilbert space $\cK_2$) of  $D: \cD\to \cK_2$. Its domain is 
\begin{align*}
\cD^{*}
:=
\{w\in\cK_2 : \exists c=c(w)<\infty : \forall u\in\cD: \abs{\sprod{w}{ R^{-1}HR^{-1} u}}\leq c \norm{u}_2 \},
\end{align*}
and $D^*:\cD^*\to \cK_2$ is defined uniquely by the Riesz Lemma: for any $w\in\cD^*$, $D^* w$ is the unique element of $\cK_2$ such that for all $u\in \cD$
\begin{align*}
\sprod{D^* w }{u} 
=
\sprod{w}{ R^{-1} H R^{-1} u}. 
\end{align*}
Obviously, $\cD\subset \cD^*$,  $D^*|_{\cD}=-D$, and 
\begin{align*}
D
\prec 
D^{**}
\preceq
-D^{*}
\end{align*}
In order to conclude 
\begin{align*}
D^{**}
=
-D^{*}
\end{align*}
and thus essential skew-self-adjointness of $D$, as defined in \eqref{op-D-def}, it is sufficient to prove that $D^*$ is skew-symmetric on $\cD^*$. 

For $K<\infty$, let 
\begin{align*}
&
r^{K}_{k}(\omega)
:=
r_{k}(\omega) \ind{K^{-1}\leq r_{k}(\omega)\leq K},
&&
h^{K}_{k,l}(\omega)
:=
h_{k,l}(\omega) \ind{\abs{h_{k,l}(\omega)}\leq K}.
\end{align*}
These truncated functions inherit the conductance symmetries \eqref{s b symm} and stream-tensor (anti)symmetries \eqref{h-tensor}. We define the \emph{bounded} operators $R^{K}, (R^{K})^{-1}, H^{K}: \cV_2\to\cV_2$ by the formulas \eqref{basic-ops}, \eqref{R-op}, \eqref{H-op}, with the functions  $r_k$, $h_{k,l}$ replaced by their truncated versions $r^{K}_k$, $h^{K}_{k,l}$. 

\begin{lemma}
\label{lem: weak-limit}
For $w\in\cD^*$ 
\begin{align}
\label{weak-lim}
D^*w 
=
-
\wlim_{K\to\infty}  \Pi \, (R^{K})^{-1} \,  H^{K} \,  (R^{K})^{-1} \, w,
\end{align}
where $\wlim$ denotes weak limit in the Hilbert space $(\cK_2, \norm{\cdot}_2)$.
\end{lemma}

\begin{proof}
This is straightforward. Let $w\in\cD^*$ and $u\in\cD$. Then 
\begin{align*}
-
\lim_{K\to\infty}
\sprod{  (R^{K})^{-1} \,  H^{K} \,  (R^{K})^{-1}  w}{u}
& 
=
\lim_{K\to\infty}
\sprod{w}{  (R^{K})^{-1} \,  H^{K} \,  (R^{K})^{-1}  u}
\\[10pt]
& 
=
\sprod{w}{  R^{-1} \,  H \,  R^{-1}   u}
=
\sprod{D^* w}{u}.
\end{align*}
The first step is legitimate, since for $K<\infty$, the operator $(R^{K})^{-1} \,  H^{K} \,  (R^{K})^{-1}$ is bounded and skew-self-adjoint. The second step follows from the fact that (due to the choice of the core $\cA$) for any $u = R \nabla g$, $g\in\cA$, there exists a $K_0<\infty$, such that for any $K>K_0$, $(R^{K})^{-1} \,  H^{K} \,  (R^{K})^{-1}  u = R^{-1} \,  H \,  R^{-1}   u$. 

Finally, since $\cK_2=\cD^{{\tt cl}2}$, \eqref{weak-lim} follows. 
\end{proof}

From \eqref{weak-lim} the skew-symmetry of the operator $D^*:\cD^*\to\cK_2$ drops out: for  $u,w\in\cD^*$
\begin{align*}
\sprod{w}{D^* u}
&
=
\lim_{K\to\infty}
\sprod{w}{ (R^{K})^{-1} \,  H^{K} \,  (R^{K})^{-1}  u}
\\[10pt]
&
=
-
\lim_{K\to\infty}
\sprod{ (R^{K})^{-1} \,  H^{K} \,  (R^{K})^{-1}  w}{ u}
=
-
\sprod{D^* w}{ u}.
\end{align*}
This concludes the proof of essential skew-self-adjointness of the operator $D:\cD\to \cK_2$. We also conclude essential skew-self-adjointness of the operator $B= \Lambda^* D \Lambda$ on the core $\cC=\Lambda^* \cD$. 
\end{proof}

This also concludes checking all conditions of Theorem \ref{thm: rsc} in the concrete setting and thus also the proof of Theorem \ref{thm: main}. 
\qed

\subsection{Bonus: Harmonic coordinates - existence and uniqueness}
\label{ss: Harmonic coordinates - existence and uniqueness}

\begin{proposition}
[Existence and uniqueness of harmonic coordinates]
\label{prop: harmonic coordinates}
\phantom{}
\\
Assume 
{\red \eqref{r-l2}}
and 
{\blue \eqref{rrec-l2}}. 
Given $\varphi\in \cH_{-}$ there exists a unique solution $w\in R^{-1}\cK_2 \buildrel {\blue \eqref{rrec-l2}} \over \subset \cU_1$ (that is: integrable gradient field) of the equation
\begin{align}
\label{hc-eq}
\sum_{k\in{\cN}}
\underbrace{(s_k(\omega)+b_k(\omega))}_{p_k(\omega)} w_k(\omega)
=
\varphi(\omega).
\end{align}
\end{proposition}

\begin{proof}
We search for a solution  $w\in R^{-1}\cK_2\subset \cU_1$. Using \eqref{sym-id} and \eqref{skew-sym-id} we write the equation \eqref{hc-eq} as 
\begin{align*}
\nabla^*(R^2 + H) w = \varphi, 
\end{align*}
and 
\begin{align*}
w=R^{-1} \Lambda g, 
\qquad
g\in \cL_2.
\end{align*}
Writing the equation for $g\in\cL_2$ we obtain
\begin{align*}
S^{1/2}\Lambda^*(I + R^{-1}HR^{-1}) \Lambda g = \varphi. 
\end{align*}
Since it is assumed that $\varphi\in\cH_{-}$ we can multiply the equation from the left with $S^{-1/2}$ and get 
\begin{align*}
(
\underbrace{\Lambda^*\Lambda}_{I_{\cL_2}} 
+ 
\underbrace{\Lambda^* R^{-1}HR^{-1} \Lambda}_{B}
) g = S^{-1/2}\varphi 
\end{align*}
Finally, since $B=-B^*$ it follows that $I+B$ is invertible (over $\cL_2$) with 
\begin{align*}
\norm{(I+B)^{-1}}_{2\to2}\leq1, 
\end{align*}
yielding the solution 
\begin{align*}
w=
\underbrace{R^{-1}\, 
\underbrace{\Lambda \, 
\underbrace{\, (I+B)^{-1} \, 
\underbrace{\, S^{-1/2} \, 
\underbrace{\, \varphi}
_{\in\cH_-}}
_{\in \cL_2}}
_{\in \cL_2}}
_{\in \cK_2}}
_{\in R^{-1}\cK_2\subset \cU_1}\,,
\end{align*}
and hence
\begin{align*}
\norm{w}_1 
\leq \norm {r^{-1}}_2 \,\norm{\varphi}_{-}
\buildrel
{\blue \eqref{rrec-l2}}
\over < \infty.
\end{align*}
\end{proof}

\section*{Appendix: Helmholtz's Theorem}
\label{s: Appendix: Helmholtz}

The content of this Appendix sheds light on the problem of when and how can a divergence-free flow/vector field expressed in divergence form, as in \eqref{h-tensor}\&\eqref{b-is-curl-of-h}. As we \emph{assumed} that this is the case for the antisymmetric part of the jump rates in \eqref{jump rates p s b}, in a formal logical sense the mathematical content of the paper does not rely on Proposition \ref{prop: Helmholtz} below.

As noted in section \ref{ss:Preliminaries}, given a stationary and ergodic, divergence-free vector field/flow 
$b:\ZZ^d\times\Omega\to \RR^{\cN}$, as, e.g. that given in \eqref{jump rates p s b},  it is a difficult and subtle issue to decide whether it can be written in divergence form, cf \eqref{h-tensor}\&\eqref{b-is-curl-of-h} or not. (Except, of course,  if it is a priori given in this form.) 

The following statement sheds light on this problem. It is an ergodic variant of what is usually referred to as \emph{Helmholtz's Theorem} from classical electrodynamics:  "In $\RR^3$, a divergence-free vector field (e.g., a time-wise stationary magnetic field) is written as the curl/rotation of a vector potential." (To be taken with a grain of salt!)

\begin{proposition}
["Helmholtz's Theorem"]
\label{prop: Helmholtz}
\phantom{}

\begin{enumerate} [(i)]

\item 
Assume  $b\in\cL_1 $. 
There exists a stream tensor field $H: \ZZ^d\times\Omega \to \RR^{\cN\times\cN}$ with $\displaystyle H_{k,l}(x, \cdot)\in \cL_{1{\rm w}}
:=\{f\in\cL: \norm{f}_{1{\rm w}}:=\sup_{0<\lambda<\infty} \lambda \pi(\abs{f}>\lambda) <\infty \}$, 
\begin{align}
\label{H-stream}
H_{-k,l}(x+k,\omega)
\!=\!
H_{k,-l}(x+l,\omega)
\!=\!
H_{l,k}(x,\omega)
\!=\!
-H_{k,l}(x, \omega)
\end{align}
with stationary increments
\begin{align}
\label{H-stincr}
H_{k,l}(y,\omega)
-
H_{k,l}(x,\omega)
=
H_{k,l}(y-x,\tau_y\omega)
-
H_{k,l}(0,\tau_x\omega)
\black,
\end{align}
such that Helmholtz's relation holds
\begin{align}
\label{helmholtz}
b_k(x,\omega)
=
\sum_{l\in\cU} H_{k,l} (x,\omega)
=
\frac12
\sum_{l\in\cU} (H_{k,l} (x,\omega)-H_{k,l} (x-l,\omega))
=
\sum_{l\in\cU} H_{k,l} (0,\tau_x\omega)
\end{align}

\item 
If $b\in\cL_p$, $p\in(1,2]$ than the same hold with $H_{k,l}(x, \cdot)\in\cL_p$. 

\item 
Assume $b\in\cL_2$. 
There exists a \emph{stationary} stream field $h:\ZZ^d\times\Omega\to \RR^{cN\times\cN}$ with $h_{k,l}(x , \omega)=h_{k,l}(\tau_x\omega)$, $h_{k,l}\in \cL_2$ such that 
\begin{align}
\label{stationary stream}
H_{k,l}(x,\omega)
=
h_{k,l}(x,\omega)
-
h_{k,l}(0,\omega)
\end{align}
if and only if $b_k\in\cH_{-1}(\abs{\Delta}):=\{f\in\cL_2: \displaystyle \lim_{\lambda\searrow0} \sprod{f}{ (\lambda I-\Delta)^{-1}f}<\infty\}$. 
\end{enumerate}

\end{proposition}

\noindent
As we \ul{assume} \eqref{h-tensor}\&\eqref{b-is-curl-of-h}) the mathematical content of this paper formally does not rely on
roposition \ref{prop: Helmholtz}. Therefore, rather than presenting its full proof we only give some hints. 

\begin{proof}
[Hints to some elements of the proof of Proposition \ref{prop: Helmholtz}]
\phantom{}

\smallskip
\noindent
The $\cL_2$ setting of part (ii) and part (iii) appear as Proposition 11 in \cite{kozma-toth-17}, with complete proofs.

\smallskip
\noindent
The $\cL_1$ setting of part (i) is more subtle. From an appropriate adaptation of the Calder\'on-Zygmund Decomposition (see e.g. \cite{stein-70}) to the ergodic context $(\Omega, \cF, \pi, \tau_x:\Omega\to\Omega: x\in\ZZ^d)$, it follows that for $f\in\cL_1$, $\norm{\abs{\Delta}^{-1}\partial_l\partial_k f}_{1{\rm w}}\leq C \norm{f}_1$, $k,l\in\cN$, where the constant $C$ depends only on the dimension $d$. Once this is established, the arguments from the proof of Proposition 11 in \cite{kozma-toth-17} yield the result. A similar (though, not identical) adaptation of the Calder\'on-Zygmund theory to a random conductance setting appears in \cite{biskup-salvi-wolff-14}.

\smallskip
\noindent
Given (i) in the $\cL_1\to\cL_{1{\rm w}}$ setting, and (ii) in the $\cL_2\to\cL_{2}$ setting, (ii) in the $\cL_p\to\cL_{p}$, $1<p<2$ setting follows by Marcinkiewicz interpolation, cf \cite{stein-70}. 

\end{proof}

\bigskip

\noindent
{\bf \large Acknowledgements:}
Discussions with Marek Biskup, Gady Kozma, Stefano Olla and Christophe Sabot about the content of this paper were always instructive and joyful. 
\\
This work was supported by the Hungarian National Research and Innovation Office (NKFIH) through the grants K-143468 and ADVANCED-152922.

\vskip2cm

\hbox{
\vbox{\hsize=15cm\noindent
{\sc B\'alint T\'oth}
\\
Alfr\'ed R\'enyi Institute of Mathematics, Budapest
\\
email: {\tt toth.balint@renyi.hu}
}
}

\end{document}